\documentclass[11pt,leqno]{article}

\usepackage{cite}

\usepackage{graphicx}
\usepackage{latexsym}
\usepackage{amsmath}
\usepackage{amssymb}
\usepackage{amsfonts}
\usepackage{enumerate}
\usepackage{theorem}

\theoremstyle{plain}
\newtheorem{theorem}{Theorem}[section]
\newtheorem{lemma}[theorem]{Lemma}

\newtheorem{proposition}[theorem]{Proposition}

{\theorembodyfont{\rmfamily}
\newtheorem{defin}[theorem]{Definition}
\newtheorem{remark}[theorem]{Remark}
\newtheorem{prob}[theorem]{Problem}

}

\newcommand{\cvd}{\hfill$\square$}

\newenvironment{proof}[1][Proof]{\noindent\textbf{#1} }{\cvd}
\renewcommand{\eqref}[1]{\textnormal{(\ref{#1})}}

\numberwithin{equation}{section}

\newcommand{\rmi}{\mathrm{i}}

\newcommand{\R}{\mathbb{R}}
\newcommand{\N}{\mathbb{N}}

\newcommand{\divrg}{\mathrm{div}\,}

\title{The stability for the Cauchy problem for elliptic equations\thanks{This work has been completed  in spite of the indiscriminate budget cuts of the Italian Ministry of University and Research.}}

\author{Giovanni Alessandrini,\thanks{Dipartimento di Matematica e Informatica,
Universit\`a degli Studi di Trieste, via Valerio 12/1, 34127
Trieste, Italy. E-mail: \textsf{alessang@units.it}}
\ Luca Rondi,\thanks{Dipartimento di Matematica e Informatica,
Universit\`a degli Studi di Trieste, via Valerio 12/1, 34127
Trieste, Italy. E-mail: \textsf{rondi@units.it}} \ Edi
Rosset,\thanks{Dipartimento di Matematica e Informatica,
Universit\`a degli Studi di Trieste, via Valerio 12/1, 34127
Trieste, Italy. E-mail: \textsf{rossedi@units.it}} \ and
Sergio Vessella\thanks{DIMAD, Universit\`a degli Studi di Firenze,
via Lombroso 6/17, 50134 Firenze, Italy. E-mail:
\textsf{sergio.vessella@dmd.unifi.it}}}

\date{}

\begin{document}

\setcounter{section}{0}
\setcounter{secnumdepth}{2}

\maketitle

\begin{abstract}
We discuss the ill-posed Cauchy problem for elliptic equations, which is pervasive
in inverse boundary value problems modeled by elliptic equations.

We provide essentially optimal stability results, in wide generality and under substantially minimal assumptions.

As a general scheme in our arguments, we show
that all such stability results can be derived by the use of a single building brick,
the three-spheres inequality.

\medskip


\medskip

\end{abstract}

\section{Introduction\label{Sect 1}}

Hadamard, in his paper of 1902 \cite{hadamWP} where he laid the basis of the notion
of well-posed problems, used the Cauchy problem for Laplace's equation as
his first example of a problem which is not well-posed. Later, in 1923, he
published \cite{hadamCAU} his well-known example of instability, see Subsection~\ref{Subsect 1.1} below for further discussion and also Maz$'$ya and Shaposhnikova  \cite{mazya} for additional information.

It may be curious to note that in the same span of time, the physical
relevance of this problem was to be encountered in the applications. For
instance, in geophysical underground prospection, the geoelectrical method was initiated
in those years, see for instance Stefanesco et al. \cite{stefanesco} and the historical
account
in Zhdanov and Keller \cite{zhdanov}. And in fact it is well acknowledged, by now,
that the geoelectrical method involves, even in its most basic formulation,
the solution of a Cauchy problem for Laplace's equation! On this respect,
one can consult, for instance, the initial considerations in the
book by Lavrent$'$ev, Romanov and {\v{S}}i{\v{s}}atski{\u\i}
\cite{LRS}.

Nowadays, it is widely recognized that the Cauchy problem for Laplace's
equation, and more generally for elliptic equations, has a central position
in all inverse boundary value problems which are modeled by means of
elliptic partial differential equations, Inverse Scattering,
Electrical Impedance Tomography, Optical Tomography, just to mention a
few. The continuing interest on this kind of problem is documented by the
number of publications which are currently appearing on this problem. For
instance, we have recorded at least 15 papers explicitly devoted to this
topic, which have appeared in the last three years on this Journal.

Therefore, we believe that it might be useful to formulate in a clear
fashion the state of the art on the issue of stability, which is obviously a
crucial cornerstone of the convergence analysis of any reconstruction
procedure and also of the stability analysis of many \emph{nonlinear}
inverse boundary value problems, whose treatment involves, in one way or
another, the analysis of an ill-posed Cauchy problem.

In this introduction we do not intend to present a complete discussion on
the historical development on this subject since Hadamard, because various
monographs, Lavrent$'$ev \cite{lavrentev}, Payne \cite{paynelibro}, Lavrent$'$ev, Romanov and {\v{S}}i{\v{s}}atski{\u\i}
\cite{LRS}, Isakov \cite{isakovlib1,isakovlib2}, H\"ormander \cite{hormanderbook1,hormanderbook2}, already contain abundant information on such
development. However, besides stating and proving results of stability of
sufficient generality and optimality it may be useful to discuss some
different, although intertwined, lines of reasoning which, in our view, have
led to the current state of the art.

\subsection{Instability and conditional stability\label{Subsect 1.1}}

In his essay of 1923 \cite{hadamCAU}, Hadamard provided a fundamental example which
shows that a solution of a Cauchy problem for Laplace's equation \emph{does
not} depend continuously upon the data. The example is as follows.

Consider the solution $u=u_{n}$, $n=1,2,\ldots$ to the Cauchy problem
in the upper half plane
$$
\left\{
\begin{array}{ll}
\Delta u=0, & \text{in }\{(x,y)\in\mathbb{R}^{2}\ |\
y>0\},\\
u(x,0) =0, & \text{for every }x\in \mathbb{R},\\
u_y(x,0)=A_n\sin nx, & \text{for every }x\in \mathbb{R}.
\end{array}
\right.
$$
We have
$$
u_n=\frac{A_n}{n}\sin nx\sinh ny\text.
$$
If we choose $A_n=\dfrac{1}{n}$ or $A_n=\dfrac{1}{n^{p}}$ for some $p>0$,
or even $A_n=e^{-\sqrt{n}}$, it turns out that
$$
u_{n,y}(x,0) \to 0\text{ uniformly as }n\to
\infty
$$
whereas, for any $y>0$,
$$
u_n(x,y) =\frac{A_n}{n}\sin nx\sinh ny\text{ blows up as }
n\to \infty.
$$
As is well-known since Tikhonov \cite{tichonov} the modern notion of \emph{stability}
for ill-posed problems (also called \emph{conditional stability}) consists
of estimating the dependence upon the data of the unknown solution of the
problem at hand, when an a-priori bound on the solution itself is available.

It is a remarkable fact that Hadamard eventually acknowledged that
continuous dependence can be restored in presence of an a-priori bound. In
fact, in his treatise of 1964 \cite[p.~146]{hadamPDE} he wrote
\begin{quote}``D'apr\`{e}s un
remarquable r\'{e}sultat d\^{u} a M. Pucci, l'absence de
continuit\'{e} de la solution $u$ du probl\'{e}me de Cauchy
consid\'{e}r\'{e}e comme functionelle des donn\'ees initiales est
solidaire du fait que cette solution est susceptible d'augmenter
infinitement: M. Pucci constate que le choses changent si l'on
connait une borne sup\'{e}riore de le valeur absolute
$|u|$.''\footnote{``After a remarkable result due to Mr. Pucci,
the lack of continuity of the solution $u$ of the Cauchy problem
considered as a functional of initial data is joint to the fact
that this solution is susceptible of infinite growth: Mr. Pucci
observes that things change if one knows an upper bound of the
absolute value $|u|$.''}\end{quote} In fact, Hadamard is referring
to a paper by Pucci \cite{puccilincei} of 1955 where one of the
first results of stability for the Cauchy problem for Laplace's
equation was obtained. Let us recall that in the same years other
stability estimates were obtained by John \cite {john1}, Landis
\cite{landiscauchy} and Lavrent$'$ev
\cite{lavrentevdokl,lavrentevisz}, see also Pucci \cite{pucci2}
and John \cite{john2}.

It may be instructing to observe that the same example by Hadamard
may be used to exhibit the best possible rates of continuous
dependence in presence of an a-priori bound. To this purpose let
us describe a Cauchy problem in the most simple and favourable
setting. We express the a-priori bound and the bounds on the data
with respect to norms which can be considered as the natural ones
in the standard variational formulation of Laplace's equation, but
it will be evident that analogous results would be obtained also
if other (reasonable) functional frameworks are considered.

Consider a Cauchy problem in a rectangle
\begin{equation}\label{Int1}
\left\{
\begin{array}{ll}
\Delta u=0, & \text{in }(0,\pi) \times ( 0,1),\\
u(x,0) = 0, & \text{for every }x\in (0,\pi),\\
u_y(x,0) =\psi (x), & \text{for every }x\in (0,\pi),
\end{array}
\right.
\end{equation}
and, in order to make things even simpler, we further assume a zero
Dirichlet condition on the vertical sides of the rectangle
\begin{equation}\label{Int2}
u(0,y) =u(\pi,y) = 0,\quad \text{for every }y\in(0,1).
\end{equation}
The natural function space of the solution $u$, in a variational setting, is the Sobolev
space $H^{1}\left( \left( 0,\pi \right) \times \left( 0,1\right) \right) $
and thus, as a-priori information on the unknown solution $u$, we assume the
bound on the Dirichlet integral
\begin{equation}
\iint_{\left( 0,\pi \right) \times \left( 0,1\right) }\left(
u_{x}^{2}+u_{y}^{2}\right) dxdy\leq E^{2}\text{,}  \label{Int3}
\end{equation}
for a given $E>0$.

The prescribed inhomogeneous data $\psi $, which express the (partial)
Neumann data on the lower horizontal side of the rectangle, naturally lives in
the trace space $H^{-\frac{1}{2}}\left( 0,\pi \right) $. Let us assume then that the
following error bound is known
\begin{equation}
\left\Vert \psi \right\Vert _{H^{-\frac{1}{2}}\left( 0,\pi \right) }\leq \eta
\text{,}  \label{Int4}
\end{equation}
for some given $\eta >0$.

A stability estimate would consist of a bound of some norm of $u$ evaluated
inside the rectangle in terms of some function $\omega \left( \eta
,E\right) $ which should be infinitesimal as $\eta \to 0^+$. The
Hadamard example provides us with limitations on such infinitesimal rate.
Let us choose once more
\begin{equation}
\psi _{n}\left( x\right) =A_{n}\sin nx\text{, }n=1,2,\ldots  \label{Int5}
\end{equation}
and let us select $A_{n}$ in such a way that equality holds in (\ref{Int3}).
We obtain
\begin{equation}
A_{n}^{2}=\frac{2}{\pi }\frac{2n}{\sinh 2n}E^{2}\text{.}  \label{Int6}
\end{equation}
Consequently, in (\ref{Int4}) we have equality when $\eta
=\eta _{n}$ where $\eta _{n}$ is given by
\begin{equation}
\eta _{n}^{2}=E^{2}\frac{2}{\sinh 2n}\sim 4E^{2}e^{-2n}\text{, as }
n\to \infty \text{.}  \label{Int7}
\end{equation}
If we wish to estimate the $L^{2}$-norm of $u$ in the rectangle $\left(
0,\pi \right) \times \left( 0,T\right) $, for some $T\in \left( 0,1\right] $, then we see that the solution to (\ref{Int1}), (\ref{Int2}) with $\psi $
given by (\ref{Int5}), (\ref{Int6}), satisfies
\begin{equation}
\left\Vert u_{n}\right\Vert _{L^{2}\left( \left( 0,\pi \right) \times \left(
0,T\right) \right) }^{2}=\frac{E^{2}}{n\sinh 2n}\left( \frac{\sinh 2nT}{2n}
-1\right) \sim E^{2}\frac{e^{2n\left( T-1\right) }}{2n^{2}}\text{, as }
n\to \infty \text{.}
\label{Int8}
\end{equation}
That is
\begin{equation}
\left\Vert u_{n}\right\Vert _{L^{2}\left( \left( 0,\pi \right) \times \left(
0,T\right) \right) }\sim \frac{E}{\sqrt{2}}\left( \frac{\eta _{n}}{2E}\right) ^{\left(
1-T\right) }\left( \log \frac{2E}{\eta _{n}}\right) ^{-1}\text{, as }
n\to \infty \text{.}  \label{Int9}
\end{equation}
Therefore, if $T<1$, then the stability of the determination of $u$ \emph{up to the level}
$y=T$ is \emph{at best }of \emph{H\"{o}lder} type.
Whereas, if we want to recover $u$ in all of its domain of definition (up to
the top side of the rectangle, \emph{where no boundary data is prescribed})
then the best possible rate of stability is \emph{logarithmic}.

It is generally acknowledged that this phenomenon has a quite general
character when dealing with the Cauchy problem for elliptic equations. And
in fact the results in the following sections agree with such a scheme.

We shall distinguish between two types of results.

\begin{description}
\item{\textbf{Stability estimates in the interior.}} A solution $u$
of a Cauchy problem is
a-priori
known to be bounded (with respect to some norm) on a  connected open set
$\Omega $,
Cauchy data are prescribed on some portion $\Sigma $ of $\partial \Omega $
and we wish to estimate $u$ on some connected open subset $G$ of $\Omega $
which is at a
positive distance from $\partial \Omega \setminus \Sigma $, the part of the
boundary where no data are prescribed.

\item{\textbf{Global stability estimates.}} We want to estimate $u$ in some
norm in all of $\Omega $ when an upper bound with a slightly stronger norm
in the same set $\Omega $ is a-priori known, and, as before, Cauchy data
are prescribed on a portion $\Sigma $ of $\partial \Omega $.
\end{description}

We reiterate that the Hadamard example tells us that for a \emph{stability
estimate in the interior} we cannot expect anything better than a \emph{H\"{o}lder rate}, whereas for \emph{global stability} the optimal rate will
be of \emph{logarithmic type} at most.

A rather general treatment of stability in the interior is due to Payne
\cite{payne70} (see also Payne \cite{payne60}, Trytten \cite{trytten}). Global stability estimates in a
wide generality, that is for general elliptic operators and general domains,
have been known and used for quite a while, but probably statements and
proofs are not easily available in the literature. In fact, a surge of
interest on this topic occurred in the 90's in connection with nonlinear
inverse boundary value problems with unknown boundaries and global stability
estimates were described and used (Alessandrini \cite{alesscrack93}, Alessandrini and
Di Benedetto \cite{alessdib},
Beretta and Vessella \cite{beve}, Bukhgeim, Cheng and Yamamoto \cite{bukhgchengyama2,bukhgchengyama1}, Alessandrini, Beretta, Rosset and Vessella \cite{ABRVpisa},
Cheng, Hon and Yamamoto \cite{chenhonyama1,chenhonyama2}) but, unfortunately, most of the times such
estimates were not explicitly stated as independent results. As an
exception, we mention Takeuchi and Yamamoto \cite[Theorem~10]{takyamamoto}.

\subsection{Analytic continuation\label{Subsect2}}

In the special case of two space variables, the Cauchy problem for Laplace's
equation is equivalent to the problem of continuation of a complex analytic
function from values prescribed on an arc. This problem was treated with
great ingenuity by Carleman in 1926. His theory is expounded in the essay of 1926
\cite{carlemanlibro} and his ideas have had a great influence in the subsequent
developments of the theory. A modern treatment of the connection between the
Cauchy problem and the analytic continuation in two dimensions can be found
in \cite[Ch.~1, \S~2]{LRS}. The crucial tool for stability is the so-called
method of harmonic measure. In fact, the seminal idea of this method can be
traced back to Carleman \cite{carleman21}, but a general formulation of this approach can be
attributed to F. and R.~ Nevanlinna \cite{nevanlinna}, see in this respect Goluzin
\cite[Ch.~VII, \S~4]{goluzin}.

The two-dimensional theory for Cauchy problems maintains a special position
also when dealing with elliptic operators with variable coefficients.
Indeed, still with the aid of complex analytic methods, uniqueness, Alessandrini and Magnanini \cite{alessmagn}, and
stability, Alessandrini and Rondi \cite{alerondiSIAM}, can be obtained \emph{with no need }of regularity assumptions on
the coefficients (contrary to what happens in higher dimensions, see in
this respect Subsection~\ref{Subsect1.3} below). An extended version of the
harmonic measure technique has been developed also for the variable
coefficients case in Alessandrini and Rondi \cite{alerondiSIAM}. In Section~\ref{2Dsec} below we shall
discuss in more detail this kind of results.

\subsection{The Cauchy problem and the unique continuation property\label{Subsect1.3}}

An issue which is strictly related to the Cauchy problem is the one of
\emph{unique continuation}. An elliptic operator $\mathcal{L}$ is said to have the
(weak) unique continuation property if for any solution $u$ to $\mathcal{L}u=0$ in a
connected open set $\Omega \subset \mathbb{R}^{n}$, that vanishes on an
open subset $G\subset \Omega $, it follows $u\equiv 0$ in $\Omega $. A
general proof of the equivalence of the \emph{uniqueness} of the Cauchy
problem with the weak \emph{unique continuation property} can be traced
back to Nirenberg \cite{nirenberg}. This equivalence has been especially important in
establishing the limits of validity of uniqueness in terms of the regularity
of the coefficients of the elliptic operator involved. In fact it was shown
by Pli\v{s} \cite{plis} that for an elliptic operator $\mathcal{L}$ in dimension $n\geq 3$, the unique continuation may fail if
the coefficients of the principal part are H\"{o}lder continuous of any
exponent smaller than $1$. Further examples were obtained by Miller \cite{miller1}, see also
Miller \cite{miller2} for improvements and discussion on such counterexamples, and
Mandache \cite{mandache}, Filonov \cite{filonov} for further developments.

The progress on this issue of unique continuation was initiated by Carleman \cite{carleman33}, subsequent advances were due to M\"uller \cite{muller},
Heinz \cite{heinz}, Hartman and Wintner \cite{hartmanw}, Cordes \cite{cordes},
Aronszajn \cite{aronszajn}.
Eventually, it was
proved by Aronszajn, Krzywicki and Szarski
\cite{AKS} that the unique continuation property
holds true when the coefficients in the principal part are Lipschitz
continuous. Soon afterwards, Pli\v{s} produced
the already mentioned example that shows that such Lipschitz continuity
provides indeed the crucial threshold.
A great deal of investigation followed, especially with the
purpose of extending the unique continuation property to Schr\"{o}dinger
operators with singular potentials. With no ambition of completeness, let us
mention H\"{o}rmander \cite{hormander}, Jerison and Kenig \cite{jerisonkenig},
Garofalo and Lin \cite{garofalolin1,garofalolin2},
Fabes, Garofalo and Lin \cite{fabesgarofalolin} and, more recently, Koch and Tataru \cite{kochtataru}.

\subsection{Three-spheres inequalities\label{Subsect1.4}}

The unique continuation property is also connected to the problem of the
stability for the Cauchy problem. In fact, proofs of the unique continuation
property depend on inequalities which can also be applied to the estimation
of stability. There are two families of such inequalities

\begin{itemize}
\item Carleman estimates,

\item Three-spheres inequalities.
\end{itemize}

Both types have been succesfully used in the study of stability,
and they are strictly intertwined, in fact three-spheres inequalities can be
deduced by Carleman estimates.

In this paper, a central theme that we intend to stress is that
three-spheres inequalities can be used as a universal building brick to
derive optimal stability estimates. Our guiding idea shall be to confine all
the hard-analysis which is required to the derivation of a \emph{basic inequality} in
a simple geometrical setting (the three-spheres) and then use it
iteratively to adapt to general geometrical configurations.

We find quite instructing to remark at this point that Hadamard, the same
mathematician who first pointed out the \emph{ill} character
of the Cauchy problem, was the one who first provided a \emph{cure}
(actually, even before the illness was diagnosed!). In fact, Hadamard first
stated in 1896 \cite{hadam3s} a \emph{three-circles inequality}, which in its
simplest manifestation is as follows. Given a holomorphic function $f$ in
the disk $\left\{ z\in \mathbb{C}\mid \left\vert z\right\vert <R\right\} $
then the function
\[
\log r\to \log \left( \max\limits_{\left\vert z\right\vert
=r}\left\vert f\left( z\right) \right\vert \right) \text{, }0<r<R
\]
is convex. Quoting once more Hadamard \cite[p.~94]{hadam3s}
\begin{quote}``\ldots je d\'esignerai par $\eta$
le logarithme du module maximum de la fonction sur le cercle de rayon
$\varepsilon^{\xi}$ (o\`u $\xi$ est un nombre r\'eel quelconque). Le
lieu du point $(\xi,\eta)$ est une courbe $C$ qui tourne toujours sa concavit\'e vers le $\eta$ positifs; \ldots''\footnote{``\ldots I shall denote with $\eta$ the logarithm of the maximum modulus of the function on the circle of radius
$\varepsilon^{\xi}$ (where $\xi$ is any real number). The locus of the points
$(\xi,\eta)$ is a curve $C$ that always bends its concavity
towards the positive $\eta$'s; \ldots''}\end{quote}
In other terms, if $0<r_{1}<r_{2}<r_{3}<R$, then
\[
\max\limits_{\left\vert z\right\vert =r_{2}}\left\vert f\left( z\right)
\right\vert \leq \left( \max\limits_{\left\vert z\right\vert
=r_{1}}\left\vert f\left( z\right) \right\vert \right) ^{\alpha }\left(
\max\limits_{\left\vert z\right\vert =r_{3}}\left\vert f\left( z\right)
\right\vert \right) ^{1-\alpha }\text{,}
\]
where $\alpha \in \left( 0,1\right) $ is given by
\[
\alpha =\frac{\log \frac{r_{3}}{r_{2}}}{\log \frac{r_{3}}{r_{1}}}\text{.}
\]
This inequality had a great influence in the following development of
complex analysis \cite{duren} but it had also a seminal character in the study of
unique continuation for elliptic equations. A three-spheres inequality for
elliptic operators whose principal part coefficients are $C^{2}$ was proved
by Landis \cite{landis3s}, and in fact his proof was based on Carleman's type
estimates. Another proof obtained by a method of differential inequalities
for integral norms, which took the name of \emph{logarithmic convexity},
was obtained by Agmon \cite{agmon}. In the 70's the general concept of logarithmic
convexity had indeed a notable influence in the analysis of various
ill-posed problems for partial differential equations. The proceedings book
edited by Knops \cite{knops} documents the advances in this direction.

With more precise connection with the stability for the Cauchy problem, this
approach reached its apex in the work by Payne \cite{payne60,payne70}, see also the
contribution by Trytten \cite{trytten}. In particular in \cite{payne70}, Payne obtained
stability estimates in the interior of H\"{o}lder type when the coefficients
of the principal part are $C^{1}$.

Returning however to the more specific issue of the three-spheres
inequalities, we mention that recent proofs have been obtained by
Brummelhuis \cite{brummelhuis}, and by Kukavica \cite{kukavica}. Both authors use, with some
variations, the so-called method of the \emph{frequency function} by
Garofalo and Lin \cite{garofalolin1}. In fact the \emph{frequency function }method can
be viewed as a further notable advance and clarification of the ideas of
logarithmic convexity.

We shall formulate various versions of the
three-spheres inequality and we shall provide a proof which is modeled, with
few adjustements, on the one by Kukavica \cite{kukavica}. Let us remark however that the
same inequality might be obtained also by means of Carleman estimates. We
refer to Vessella \cite{vessella} for a general discussion of this approach, in the
wider context of parabolic equations. Let us also quote Lin, Nakamura and Wang \cite{linnakawang}
for a very recent investigation in this direction.

\subsection{Doubling inequalities \label{Subsect1.5}}

It cannot be omitted at this point that reasearch on unique continuation has
been especially concentrated on the aspect of the \emph{strong} unique
continuation property, that is, if $u$ solves the elliptic equation $\mathcal{L}u=0$
in a connected open set $\Omega $ and $u$ vanishes of infinite order at one
point $x_{0}\in \Omega $ (that is $u\left( x\right) =O\left( \left\vert
x-x_{0}\right\vert ^{N}\right) $ as $x\to x_{0}$, for every $N=1,2,\ldots$) then $u$ has to be zero everywhere.

This property does not have a direct connection with the stability of the
Cauchy problem. However, quantitative versions of the strong unique
continuation property have shown to be very useful in the study of stability
of certain inverse boundary value problems. Such quantitative estimates are
in fact the \emph{doubling inequality} by Garofalo and Lin \cite{garofalolin1} and the \emph{doubling
inequality at the boundary}, see Adolfsson, Escauriaza and Kenig \cite{adolescaken},
Adolfsson and Escauriaza \cite{adolesca}, Kukavica and Nystr\"om \cite{kuknyst}. Applications to elliptic
inverse boundary value problems occurred for instance in Alessandrini, Rosset and Seo \cite{alrossseo},
Alessandrini, Beretta, Rosset and Vessella \cite{ABRVpisa}.
We recall also that such quantitative estimates have been used in connection with the
problem of estimates of continuation from measurable sets, which is very much related to the Cauchy problem as well. In this respect let us mention Nadirashvili \cite{nadir1,nadir2}, Vessella \cite{sergio1,sergio2}, and Malinnikova \cite{malinnikova}.

\subsection{The scheme of a stability proof\label{Subsect1.6}}

Our general scheme of proof for the stability of a Cauchy problem will be
as follows.

\begin{enumerate}[I)]
\item First we prove a three-spheres inequality. We shall need an inequality of
this kind not only for solutions of homogeneous elliptic equations but more
generally for solutions of inhomogeneous equations $\mathcal{L}u=\mathcal{F}$ with an $H^{-1}$
right-hand side $\mathcal{F}$ (see Theorem~\ref{cor:3spheres_nonhomogeneous_2}).

\item Next we use iteratively the three-spheres inequality to obtain \emph{
estimates of propagation of smallness}. We assume an a-priori bound on a
solution $u$ on its domain $\Omega $ and that $u$ is small in a given
(small) ball $B_{r_{0}}\left( x_{0}\right) \subset \Omega $ and we estimate
how small is $u$ in some larger connected open set $G\subset \Omega $. If $G$ is at
a positive distance from $\partial \Omega $ we shall speak of
\emph{estimates of propagation of smallness in the interior} (see Theorem~\ref{theo:PSinterior}),
instead if $G$ agrees with $\Omega $ we shall speak of \emph{global
estimates of propagation of smallness} (see Theorem~\ref{theo:PSglobal}
and also Remark~\ref{rem:loglog_PS}).

\item Then, given a solution of a Cauchy problem in $\Omega $ with data on a
portion $\Sigma $ of $\partial \Omega $, we extend $u$ to an open set $\mathcal{A}$
outside of $\Omega $, whose boundary agrees with $\partial \Omega $ on a
subset of $\Sigma $. We perform such extension in such a way that the
extended function $\widetilde{u}$ solves in $\mathcal{A}$ an inhomogeneous equation
with a right-hand side which is controlled in a Lipschitz fashion by the
Cauchy data on $\Sigma $ (see Theorem~\ref{theo:extension} and the following Remark~\ref{rem:extension}).

\item Finally, we apply the estimates of propagation of smallness in the
augmented domain $\widetilde{\Omega}=\overset{\circ}{\overline{\Omega \cup \mathcal{A}}}$ thus obtaining an interior stability
estimate for the Cauchy problem in Theorem~\ref{theo:Cauchyinterior}, and a global stability
estimate in Theorem~\ref{theo:Cauchy_global}.

\end{enumerate}

\subsection{The main hypotheses and statements\label{Subsect1.7}}

In order to discuss in more detail the hypotheses that are used in our treatment
it is necessary to introduce some notation and definitions.

Given $x\in \R^n$, we shall denote $x=(x',x_n)$, where $x'=(x_1,\ldots,x_{n-1})\in\R^{n-1}$, $x_n\in\R$.
Given $x\in \R^n$, $r>0$, we shall use the following notation for balls and cylinders.
\begin{equation*}
   B_r(x)=\{y\in \R^n\ |\ |y-x|<r\}, \quad  B_r=B_r(0),
\end{equation*}
\begin{equation*}
   B'_r(x')=\{y'\in \R^{n-1}\ |\ |y'-x'|<r\}, \quad  B'_r=B'_r(0),
\end{equation*}
\begin{equation*}
   \Gamma_{a,b}(x)=\{y=(y',y_n)\in \R^n\ |\ |y'-x'|<a, |y_n-x_n|<b\}, \quad \Gamma_{a,b}=\Gamma_{a,b}(0).
\end{equation*}
Throughout this paper  we shall denote by $\Omega$ a bounded
open connected subset of $\mathbb{R}^{n}$.
In places we shall assume that
the boundary of $\Omega$ is Lipschitz according to the following definition.
\begin{defin}[Lipschitz regularity]
  \label{def:Lipschitz_boundary}
We say that the boundary of
$\Omega$ is of \emph{Lipschitz class} with
constants $\rho_{0}$, $M_{0}>0$, if, for any $P \in \partial\Omega$, there exists
a rigid transformation of coordinates under which $P=0$
and
\begin{equation}
   \label{bordo_lip1}
  \Omega \cap \Gamma_{\frac{\rho_{0}}{M_0},\rho_0}(P)=\{x=(x',x_n) \in \Gamma_{\frac{\rho_{0}}{M_0},\rho_0}\quad | \quad
x_{n}>Z(x')
  \},
\end{equation}
where $Z:B'_{\frac{\rho_{0}}{M_0}}\to\R$ is a Lipschitz function satisfying
\begin{equation}
   \label{bordo_lip2}
Z(0)=0,
\end{equation}
\begin{equation}
   \label{bordo_lip3}
\|Z\|_{{C}^{0,1}\left(B'_{\frac{\rho_{0}}{M_0}}\right)} \leq M_{0}\rho_{0}.
\end{equation}
\end{defin}

\begin{remark}
  \label{rem:M0_big}
  For practical purposes it will turn out useful to assume throughout
  that $M_0\geq 1$. In fact it is evident from Definition~\ref{def:Lipschitz_boundary}
  that conditions \eqref{bordo_lip1}--\eqref{bordo_lip3} continue to hold if $M_0$ is increased.
\end{remark}

\begin{remark}
  \label{rem:normal_norm}
  Throughout this paper we shall use the convention to normalize all norms in such a way that they are
  dimensionally equivalent to their argument and coincide with the
  standard definition when the dimensional parameter $\rho_0$ equals $1$.
  For instance, the norm appearing above is meant as follows
\begin{equation*}
  \|Z\|_{{C}^{0,1}\left(B'_{\frac{\rho_0}{M_0}}\right)} =
  \|Z\|_{{L}^{\infty}\left(B'_{\frac{\rho_0}{M_0}}\right)}+
  \rho_0\|\nabla Z\|_{{L}^{\infty}\left(B'_{\frac{\rho_0}{M_0}}\right)}.
\end{equation*}

Similarly, we shall set
\begin{equation*}
\|u\|_{L^2(\Omega)}=\rho_0^{-\frac{n}{2}}\left(\int_\Omega u^2
\right)^{\frac{1}{2}},
\end{equation*}
\begin{equation*}
\|u\|_{H^1(\Omega)}=\rho_0^{-\frac{n}{2}}\left(\int_\Omega u^2
+\rho_0^2\int_\Omega|\nabla u|^2\right)^{\frac{1}{2}},
\end{equation*}
and so on for boundary and trace norms such as
$\|\cdot\|_{L^2(\partial\Omega)}$,
$\|\cdot\|_{H^{\frac{1}{2}}(\partial\Omega)}$,
$\|\cdot\|_{H^{-\frac{1}{2}}(\partial\Omega)}$.
\end{remark}

In some instances we shall require that only a limited open portion $\Sigma$ of $\partial\Omega$
be Lipschitz in the following sense.
Some further notation is necessary.
We shall denote
\begin{equation}
   \label{Sigma'}
\Sigma'=\partial\Omega\setminus \Sigma,
   \end{equation}
and, for every $P\in\Sigma$, we set
\begin{equation}
   \label{r(P)}
   r(P)=\mathrm{dist}(P,\Sigma'),
\end{equation}
\begin{equation}
   \label{rho(P)}
   \rho(P)=\min\left\{\rho_0,\frac{r(P)M_0}{\sqrt{1+M_0^2}}\right\}.
\end{equation}
\begin{defin}
  \label{def:Lipschitz_Sigma}
An open subset $\Sigma\subset\partial\Omega$ is said to be an open Lipschitz portion of $\partial\Omega$ with
constants $\rho_{0}$, $M_{0}>0$, if, for any $P \in \Sigma$, there exists
a rigid transformation of coordinates under which $P=0$
and
\begin{equation}
   \label{Sigma_lip}
  \Omega \cap \Gamma_{\frac{\rho(P)}{M_0},\rho(P)}(P)=\{x=(x',x_n) \in \Gamma_{\frac{\rho(P)}{M_0},\rho(P)}\quad | \quad
x_{n}>Z(x')
  \},
\end{equation}
where $Z:B'_{\frac{\rho_{0}}{M_0}}\to\R$ is a Lipschitz function satisfying
\eqref{bordo_lip2}--\eqref{bordo_lip3}.
\end{defin}
We shall also need an assumption on $\Sigma$ in order to control
\emph{from below} its smallness.
\begin{defin}
  \label{def:size_Sigma}
We shall say that $\Sigma$ has \emph{size at least} $\rho_1$, $0<\rho_1\leq\rho_0$, if there
exists at least one point $P\in\Sigma$ such that
\begin{equation}
   \label{rho(P)_big}
  \rho(P)\geq\rho_1.
\end{equation}
%
\end{defin}

The elliptic operators that we shall consider are of the following form
\begin{equation}
\mathcal{L}u=\mathrm{div}\left( A\nabla u\right) +cu  \label{Int16}
\end{equation}
where $A=A(x)=\left\{ a_{ij}(x)\right\}$, $x\in\mathbb{R}^n$,
is a real-valued symmetric $n\times n$ matrix
such that its entries are measurable and it satisfies, for a given constant $K\geq 1$,
the ellipticity condition
\begin{equation}
    \label{ellipticity}
    K^{-1}|\xi|^2\leq
A(x)\xi\cdot\xi\leq K |\xi|^2,\quad\hbox{for almost
every }x\in \mathbb{R}^n,\text{ for every }\xi\in{\R}^n.
\end{equation}
Furthermore, when $n\geq 3$, we also assume that, for a given constant $L>0$,
the following Lipschitz continuity holds
\begin{equation}
    \label{lipschitz}
    |A(x)-A(y)|\leq \frac{L}{\rho_0}|x-y|,\quad\hbox{for
every }x,y\in \mathbb{R}^n.\
\end{equation}
Concerning the zero order term, we assume that $c\in L^{\infty}(\mathbb{R}^n)$ with
\begin{equation}
    \label{c_bound}
    \|c\|_{L^\infty(\R^n)}\leq \frac{\kappa}{\rho_0^2}.
\end{equation}

Here and for the rest of this paper we shall assume that $A$ is a symmetric matrix of coefficients satisfying the ellipticity condition \eqref{ellipticity} and such that,
if $n\geq 3$, it also satisfies the Lipschitz condition \eqref{lipschitz}. We emphasize that, in all the following statements, whenever a constant is said to depend on $L$ (among
other quantities) it is understood that such dependence occurs \emph{only} when $n\geq 3$.

We remark once and for all that also first order terms with bounded coefficients
could be added to the treatment with little additional effort. We have
chosen to confine ourselves to the above variational structure (\ref{Int16})
because we believe that (\ref{Int16}) provides a sufficiently wide and useful
setting for enough applications. We insist however that we allow the
presence of the zero order term, with no sign, nor smallness, condition on
the coefficient $c$ because of the importance of the applications to wave
phenomena at a fixed wavenumber.

Let us now introduce the rigorous weak formulation of the Cauchy
problem that we shall use. First it is necessary to introduce some
further notation and some function spaces.

The solution $u$ will be assumed to belong to the $H^{1}\left(
\Omega \right) $ space and the following \emph{a-priori bound}
will be prescribed
\begin{equation}
\left\Vert u\right\Vert _{H^{1}\left( \Omega \right) }\leq E,
\end{equation}
for some given $E>0$, for the purpose of a \emph{global stability
estimate}, whereas for the \emph{stability in the interior} we
shall more simply require that, given $E_0>0$,
\begin{equation}
\left\Vert u\right\Vert _{L^{2}\left( \Omega \right) }\leq
E_0\text{.}
\end{equation}

We shall fix $\Sigma $ as an open connected portion of $\partial
\Omega $ and we shall always assume it to be Lipschitz.

We shall consider as test functions space the space
$H^1_{co}(\Omega\cup\Sigma)$ consisting of the functions
$\varphi\in H^1(\Omega)$ having support compactly contained in
$\Omega\cup\Sigma$. We denote by $H^{\frac{1}{2}}_{co}(\Sigma)$
the class of $H^{\frac{1}{2}}(\Sigma)$ traces of functions
$\varphi\in H^1_{co}(\Omega\cup\Sigma)$. We then define
$H^{-\frac{1}{2}}(\Sigma)$ as the dual space to
$H^{\frac{1}{2}}_{co}(\Sigma)$ based on the $L^2(\Sigma)$ dual
pairing.

The Cauchy data $g$, $\psi $ will be taken in their natural trace
spaces, namely
\begin{equation}
g\in H^{\frac{1}{2}}\left( \Sigma \right) \text{, }\psi \in
H^{-\frac{1}{2}}\left( \Sigma \right) \text{.}
\end{equation}
 As bounds on the Cauchy data we require
\begin{equation}
   \label{bound_g_psi}
\left\Vert g\right\Vert _{H^{\frac{1}{2}}\left( \Sigma \right)
}+\rho_0\left\Vert \psi \right\Vert _{H^{-\frac{1}{2}}\left(
\Sigma \right) }\leq \eta.
\end{equation}

To begin with, let us consider the following more or less standard
formulation of a Cauchy problem
\begin{equation}\label{senza F}
\left\{
\begin{array}{ll}
\mathrm{div}\left( A\nabla u\right) +cu=f, & \text{in }\Omega
\text{,} \\ u=g, & \text{on }\Sigma \text{,} \\ A\nabla u\cdot \nu
=\psi, & \text{on }\Sigma \text{,}
\end{array}
\right.
\end{equation}
where the right-hand side $f$ can be assumed in $L^{2}\left(
\mathbb{R}^{n}\right) $.

The corresponding \emph{rigorous weak formulation} of the Cauchy
problem would be to find $u\in H^{1}\left( \Omega \right) $ such
that $u_{\mid \Sigma }=g$ in the trace sense and
\begin{equation}\label{weakformsenzaF}
\int_{\Omega }\left( A\nabla u\cdot \nabla \varphi -cu\varphi
\right) =\int_{\Sigma }\psi \varphi -\int_{\Omega } f\varphi
 ,\quad\text{for every }\varphi \in
H^{1}_{co}\left( \Omega \cup\Sigma \right).
\end{equation}
Here, the integral $\int_{\Sigma }\psi \varphi $ is to be properly
interpreted as the dual pairing between $H^{-\frac{1}{2}}\left(
\Sigma \right) $ and $H_{co}^{\frac{1}{2}}\left( \Sigma \right) $.
Note that the right hand side in \eqref{weakformsenzaF} represents
a bounded linear functional over $H^1_{co}(\Omega\cup\Sigma)$. An
even more general bounded  functional over
$H^1_{co}(\Omega\cup\Sigma)$ could be written as follows
\begin{equation}\label{scriptF}
\mathcal{F}(\varphi)=\int_{\Sigma }\psi \varphi -\int_{\Omega
}\left( f\varphi -F\cdot \nabla \varphi \right) ,\quad\text{for every }\varphi \in
H^{1}_{co}\left( \Omega \cup\Sigma \right)
\end{equation}
where $F\in L^{2}\left( \mathbb{R}^{n}; \mathbb{R}^{n}\right)$ is
a given vector field with $L^{2}$ components. It will be
convenient for us to admit such type of right-hand side in our
formulation. Observe, however, that such a representation of
members of $\left(H^1_{co}(\Omega\cup\Sigma)\right)^{\ast}$ is
exceedingly non-unique (and indeed we shall fruitfully take
advantage of this non-uniqueness in
Theorem~\ref{theo:extension}).

Hence we shall prescribe
\begin{equation}\label{bound_fF}
\left\Vert f\right\Vert _{L^{2}\left( \mathbb{R}^{n}\right)
}+\frac{1}{\rho _{0}}\left\Vert F\right\Vert _{L^{2}\left(
\mathbb{R}^{n};\mathbb{R} ^{n}\right) }\leq \frac{\varepsilon
}{\rho _{0}^{2}}\text{,}
\end{equation}
for a given $\varepsilon>0$ and we shall formulate
 the \emph{rigorous weak formulation} of the Cauchy
problem as follows.

\begin{prob}[The weak formulation of the Cauchy problem]\label{weakCauchy}
To find
$u\in H^{1}\left( \Omega \right) $ such that $u_{\mid \Sigma }=g$
in the trace sense and
\begin{equation}\label{weakform}
\int_{\Omega }\left( A\nabla u\cdot \nabla \varphi -cu\varphi \right) =\int_{\Sigma }\psi \varphi -\int_{\Omega }\left( f\varphi
-F\cdot \nabla \varphi \right),\quad\text{for every }\varphi \in H^1_{co}(\Omega\cup\Sigma).
\end{equation}
\end{prob}
We remark that an \emph{abstract} interpretation of the above
stated problem could be expressed as follows

\begin{equation}\label{abstract}
\left\{
\begin{array}{ll}
\mathrm{div}\left( A\nabla u\right) +cu=\mathcal{F}, & \text{in }
\left(H^1_{co}(\Omega\cup\Sigma)\right)^{\ast} ,\\ u_{\mid \Sigma
}=g, & \text{in }H^{\frac{1}{2}}(\Sigma) ,
\end{array}
\right.
\end{equation}
where $\mathcal{F}$ is given by \eqref{scriptF}. Unfortunately,
this abstract formulation hides within its first equation a
boundary contribution which should express the Neumann condition
on $\Sigma$. Roughly speaking, such boundary contribution can be
detected,  if one \emph{formally} integrates by parts the terms
involving $\nabla \varphi$ appearing in \eqref{weakform}, and
\emph{formally} obtains
\begin{equation*}\label{conF}
\left\{
\begin{array}{ll}
\mathrm{div}\left( A\nabla u\right) +cu=f+ \mathrm{div}F, &
\text{in }\Omega \text{,} \\ u=g, & \text{on }\Sigma \text{,} \\
A\nabla u\cdot \nu =\psi+F\cdot \nu , & \text{on }\Sigma \text{,}
\end{array}
\right.
\end{equation*}
and it is evident that this formal expression incorporates as a
special case \eqref{senza F}. We reiterate, however, that the true
interpretation of the Cauchy problem that we shall use is given by
 Problem~\ref{weakCauchy}, but at the same time we emphasize that the full
strength of such a rigorous formulation shall be used at a single
step in our arguments only, and specifically in the above
mentioned Theorem~\ref{theo:extension}. In all remaining estimates
occurring in this paper we shall merely make use of the
\emph{interior} weak formulation of the elliptic equation
\begin{equation*}
\mathrm{div}\left( A\nabla u\right) +cu=f+ \mathrm{div}F ,
\end{equation*}
which can be viewed in the more customary sense of
$H^{-1}(\Omega)=\left(H^1_0(\Omega)\right)^{\ast}$.

Our main stability estimates are contained in the following theorems.

\begin{theorem}[Stability in the interior for the Cauchy problem]
    \label{theo:Cauchyinterior}
    Let $u\in H^1(\Omega)$ be a weak solution to the Cauchy Problem~\ref{weakCauchy}, where
    $\Sigma$ satisfies the conditions in Definition~\ref{def:Lipschitz_Sigma} and Definition~\ref{def:size_Sigma}, $f\in L^2(\R^n)$ and $F\in L^2(\R^n; \R^n)$ satisfy \eqref{bound_fF} and $g\in H^{\frac{1}{2}}\left( \Sigma \right)$, $\psi \in H^{-\frac{1}{2}}\left( \Sigma\right)$ satisfy \eqref{bound_g_psi}.
    There exists $\overline{h}$, $0<\overline{h}<\frac{\rho_1}{8M_0}$, with $\frac{\overline{h}}{\rho_0}$ only depending on $K$, $L$, $\kappa$, $M_0$ and $\frac{\rho_0}{\rho_1}$, such that, assuming the a-priori bound
\begin{equation}
  \label{bound_E0_stability}
  \|u\|_{L^2(\Omega)}\leq E_0,
\end{equation}
   then, for every $h$, $0<h\leq \overline{h}$, and for every connected open set $G\subset\Omega$
   such that
\begin{equation}
  \label{dist_G_bordo}
  \mathrm{dist}(G,\partial\Omega)\geq h,
\end{equation}
\begin{equation}
  \label{dist_P_G}
  \mathrm{dist}(P,G)<\frac{\rho_1}{8M_0},
\end{equation}
where $P\in\Sigma$ is the point appearing in Definition~\ref{def:size_Sigma}, we have
\begin{equation}
    \label{Cauchyinterior}
    \|u\|_{L^2(G)}\leq C(\varepsilon+\eta)^{\delta}
    (E_0+\varepsilon+\eta)^{1-\delta},
\end{equation}
where $C>0$ and $\delta\in (0,1)$ satisfy
\begin{equation}
   \label{Cdef_bis}
C= C_1\left(\frac{|\Omega|}{h^n}\right)^{\frac{1}{2}}
\end{equation}
and
\begin{equation}
\label{deltadef_bis}
\delta\geq \alpha^{\frac{C_2|\Omega|}{h^n}}
\end{equation}
with $\alpha\in (0,1)$ only depending on $K$, $L$ and $\kappa$, $C_1$ only depending on $K$, $L$, $\kappa$, $M_0$ and $\frac{\rho_0}{\rho_1}$, and
$C_2$ only depending on $K$.
\end{theorem}

\begin{remark}
Let us observe that stability with a H\"older rate could be obtained also when the connected subset $G$ of $\Omega$ is allowed to touch the boundary portion $\Sigma$,
while remaining at a positive distance from its complement $\Sigma'$. This fact will turn out to be evident from the proof, see also Remark~\ref{rem:Holder_stability_extension}.
We have chosen the present formulation because in this way a more effective evaluation
of the constants $C$ and $\delta$ is obtained in terms of the parameter $h$, which controls the distance of $G$ from the whole boundary $\partial \Omega$. It would also be possible to obtain analogous results in terms of the distance of $G$ from
$\Sigma'=\partial\Omega\setminus\Sigma$ at the price of assuming more information on the shape and regularity of the boundary of $\Sigma$ within $\partial\Omega$.
\end{remark}

\begin{theorem}[Global stability for the Cauchy problem]
    \label{theo:Cauchy_global}
    Let $\Omega$ be a connected open set of Lipschitz class according to Definition~\ref{def:Lipschitz_boundary}.
    Let $u\in H^1(\Omega)$ be a weak solution to the Cauchy Problem~\ref{weakCauchy}, where
     $f\in L^2(\R^n)$ and $F\in L^2(\R^n; \R^n)$ satisfy \eqref{bound_fF} and $g\in H^{\frac{1}{2}}\left( \Sigma \right)$, $\psi \in H^{-\frac{1}{2}}\left( \Sigma\right)$ satisfy \eqref{bound_g_psi}.
    If $u$ satisfies the a-priori assumption
\begin{equation}
  \label{bound_E_stability}
  \|u\|_{H^1(\Omega)}\leq E,
\end{equation}
   then
\begin{equation}
    \label{Cauchyglobal}
    \|u\|_{L^2(\Omega)}\leq
    (E+\varepsilon+\eta)\omega\left(\frac{\varepsilon+\eta}{E+\varepsilon+\eta}\right),
\end{equation}
where
\begin{equation}
    \label{omega_bis}
    \omega(t)\leq \frac{C}{\left(\log\frac{1}{t}\right)^\mu},\quad \hbox{for  }t<1.
\end{equation}
where $C>0$ and $\mu$, $0<\mu<1$, only depend on $K$, $L$, $\kappa$, $M_0$, $\frac{\rho_0}{\rho_1}$ and
$\frac{|\Omega|}{\rho_0^n}$.
\end{theorem}

According to the scheme already illustrated in Subsection~\ref{Subsect1.6} our first
step will be the proof of a three-spheres inequality. In fact our basic
building brick will be the following

\begin{theorem}[Three-spheres inequality]
   \label{cor:3spheres_nonhomogeneous_2}
Let $u\in H^1(B_R)$ be a weak solution to the inhomogeneous elliptic equation
\begin{equation}
    \label{div_equation_fF_tris}
    \divrg(A\nabla u)+cu=f+\divrg F, \quad\hbox{in } B_R,
\end{equation}
where
$f\in L^2(\R^n)$ and $F\in L^2(\R^n; \R^n)$ satisfy \eqref{bound_fF}. Then, for
every $r_1$, $r_2$, $r_3$, with $0<r_1<r_2<r_3\leq R$,
\begin{equation}
    \label{3spheres_nonhomogeneous_2}
    \|u\|_{L^2(B_{r_2})}\leq C \left(\|u\|_{L^2(B_{r_1})}+\varepsilon\right)^\alpha
    \left(\|u\|_{L^2(B_{r_3})}+\varepsilon\right)^{1-\alpha},
\end{equation}
where $C>0$ and $\alpha$, $0<\alpha<1$, only depend on $K$, $L$, $\kappa$, $\max\left\{\frac{R}{\rho_0},1\right\}$,
$\frac{r_2}{r_1}$ and $\frac{r_3}{r_2}$.
\end{theorem}

Our strategy for the proof of such an inequality breaks down in
the following steps.

\begin{enumerate}[a)]
\item First we prove a three-spheres inequality for the homogeneous
equation in pure principal part (Theorem~\ref{theo:3spheres_generale}) with some limitations on the radii, namely
$0<r_1<r_2<r_3/K\leq r_3\leq R$, when $n\geq 3$.
In two dimensions, a three-spheres (in fact circles) inequality may be obtained, with a
different technique, without the limitations on the radii and with a
possibly discontinuous coefficient matrix $A$ (Theorem~\ref{threespheresharmteoL2}).

\item Next we adapt the proof to operators also containing the zero
order term (Theorem~\ref{cor:3spheres_zero_order}).

\item We include the presence of the right-hand side in the equation (Theorem~\ref{cor:3spheres_nonhomogeneous}).

\item \label{step_d}  We remove the limitations on the radii, finally obtaining the above stated Theorem~\ref{cor:3spheres_nonhomogeneous_2}.
\end{enumerate}

We shall also consider estimates of propagation of smallness, in Section~\ref{sec: PS}. The main results will be
Theorem~\ref{theo:PSinterior} (propagation in the
interior) and Theorem~\ref{theo:PSglobal} (global propagation).
In fact, Theorem~\ref{theo:PSinterior} is obtained as a consequence of the
three-spheres inequality with restrictions on the radii, Theorem~\ref{cor:3spheres_nonhomogeneous}, and the above mentioned step~\ref{step_d}) is obtained as a special case of
Theorem~\ref{theo:PSinterior}.

In Section~\ref{sec: stability}, we prove the
extension argument (Theorem~\ref{theo:extension}) which enables to apply the estimates of
propagation of smallness (Theorems~\ref{theo:PSinterior} and
\ref{theo:PSglobal}) to complete the
proofs of Theorems~\ref{theo:Cauchyinterior} and \ref{theo:Cauchy_global}.

Finally, we shall address the problem of global stability when no regularity assumption
is available on the domain $\Omega$. This situation in fact often occurs in inverse problems with unknown boundaries
\cite{alesscrack93,ABRVpisa,alerondiSIAM,rondiNTA,CRV,vessella,morross1,morross2}.
If two solutions of the same equation have different domains of definition, then
their difference solves an equation in the intersection of the two domains. Such an intersection may be, in principle, highly nonsmooth, even if the starting domains satisfy some a-priori regularity assumptions. It is then useful to derive a preliminary, maybe very weak, global stability for such difference of solutions in their common domain of definition (or at least in one of its connected components). We believe that
the present formalization of this argument, expressed in Theorem~\ref{theo:Cauchyglobal_loglog}, might be a useful tool for future use in other inverse problems with unknown boundaries. Of course some assumptions will be needed anyhow. The basic one that we shall use on $\Omega$ is rather weak and can be summarized as follows. We shall assume that there exists a family $\{G_h\}$
of connected subsets invading $\Omega$ such that the measure of the difference
$|\Omega\setminus G_h|$ is controlled by a given power $h^\vartheta$
of the distance $h$ of $G_h$ from $\partial\Omega$. It may be easy to predict that, with such weak hypotheses, only a weak result can be obtained. In fact we are able to achieve only a stability of \emph{log--log}-type.


\subsection{Concluding remarks} We believe that the above theorems provide
optimal results of stability for the Cauchy problem. The optimality is
achieved in many respects.

\begin{enumerate}[1)]
\item The regularity assumptions on the coefficients of the principal part are
minimal, Lipschitz continuity when $n\geq 3$ and $L^{\infty }$ when $n=2$.
In fact these are known to be optimal conditions for uniqueness.

\item The regularity assumptions on the boundary are kept to a minimum. In
fact assuming Lipschitz regularity of the portion $\Sigma $ of $\partial
\Omega $ where the Cauchy data are assigned seems to be a nearly minimal
condition, just in order to give sense to the Cauchy problem. On the other
hand, when dealing with global stability, some conditions on the boundary,
which may ensure the uniform reachability of the boundary points from the
interior, seem necessary and Lipschitz regularity appears to be enough
general and at the same time meaningful in terms of applications. More
general reachability conditions could be considered, such as the NTA
corkscrew condition of Jerison and Kenig \cite{jerisonkenigNTA}. Let us mention in this
direction the related investigation by Rondi \cite{rondiNTA}.

\item The Cauchy data are evaluated in their natural spaces. This is indeed a
slight improvement with respect to previous studies where typically $
H^1(\Sigma) $-norm on $g$ and $L^2(\Sigma) $
-norm on $\psi $\ are considered, see Payne \cite{payne70}, Isakov \cite{isakovlib2},
Takeuchi and Yamamoto \cite{takyamamoto}.

\item As already discussed, the stability rates obtained have an optimal
character. The quantities in the stability estimates which might require
further investigation are the constants and the exponents (for instance $C$
and $\delta $ in Theorem~\ref{theo:Cauchyinterior} and $C$ and $\mu $ in Theorem~\ref{theo:Cauchy_global}). It would be
interesting, although possibly rather difficult, to simultaneously optimize
such pairs of quantities with respect to the geometry and also with respect to the
coefficients of the equation.
In connection to this issue we wish to mention
some remarkable results by Hrycak and Isakov \cite{isakovhry} and
Subbarayappa and Isakov \cite{isakovsub} who
have considered the Helmholtz equation $\Delta u+k^{2}u=0$ and proved that
the
stability estimates improve as $|k|$ increases, in the sense that, as $|k|
\to \infty$, the logarithmic term becomes negligible, whereas a
H\"older term prevails. See also Isakov \cite{isakov2008} for a related
analysis with a more general equation.
\end{enumerate}

\bigskip

The plan of the paper is as follows. Section~\ref{sec: Kucaviza} contains the
proof of a three-spheres inequality for a homogeneous equation in pure principal part,
the main result being Theorem~\ref{theo:3spheres_generale}. In Section~\ref{2Dsec}
we investigate the two-dimensional case. We illustrate the connection with quasiconformal mappings in Proposition~\ref{strmfuncprop}. We introduce the notion of $\mathcal{L}_{A}$-harmonic measure, Definition~\ref{hrmmeasdef}, and we apply it
to stability estimates in the interior for Cauchy problems, Theorems~\ref{harmmeastechthm} and \ref{harmmeastechthm2}. Finally, we deduce various forms of three-spheres (circles) inequalities, Theorem~\ref{threespheresharmteo} and Theorem~\ref{threespheresharmteoL2} which will be the one we shall use in the rest of the paper (together with Theorem~\ref{theo:3spheres_generale}).
In Section~\ref{3spheresgensec} we adapt the previously obtained three-spheres inequalities in order to encompass equations containing the zero order term and an inhomogeneous right-hand side. The final result of this section is given in Theorem~\ref{cor:3spheres_nonhomogeneous}. Section~\ref{sec: PS} is devoted to estimates of propagation of smallness. In Theorem~\ref{theo:PSinterior} we prove the interior estimate. As we already mentioned, as a corollary we also derive the proof of
the above stated Theorem~\ref{cor:3spheres_nonhomogeneous_2}. Next, in Theorem~\ref{theo:PSglobal} we prove the global estimate. Section~\ref{sec:
stability} contains the proof of the two main stability theorems for the Cauchy problem, Theorems~\ref{theo:Cauchyinterior} and \ref{theo:Cauchy_global}. As we already illustrated, Section~\ref{sec:
stability_loglog} deals with a global stability estimate for the Cauchy problem when no regularity is assumed on $\Omega$, Theorem~\ref{theo:Cauchyglobal_loglog}.

\section{The basic three-spheres inequality, 
arbitrary dimension} \label{sec:
Kucaviza}

Let us begin by considering the homogeneous elliptic equation in
pure principal part
\begin{equation}
    \label{div_equation}
    \divrg(A\nabla u)=0, \quad\hbox{in } B_R,
\end{equation}
where
$A$ satisfies \eqref{ellipticity} and \eqref{lipschitz},
for given constants $K\geq 1$ and $L>0$.

\begin{theorem}[Three-spheres inequality -- pure principal part]
    \label{theo:3spheres_generale}
       If the above sta\-ted hypotheses hold, for
every $r_1$, $r_2$, $r_3$, with $0<r_1<r_2<\frac{r_3}{K}\leq r_3\leq R$,
\begin{equation}
    \label{3spheres_generale}
    \|u\|_{L^2(B_{r_2})}\leq Q \|u\|_{L^2(B_{r_1})}^\alpha \|u\|_{L^2(B_{r_3})}^{1-\alpha},
\end{equation}
where $Q\geq 1$ only depends on $K$, $L$, $\max\left\{\frac{R}{\rho_0},1\right\}$, and where
\begin{equation}
    \label{alpha_generale}
\alpha=\frac{\log\frac{r_3}{Kr_2}}{\log\frac{r_3}{Kr_2}+C\log\frac{Kr_2}{r_1}},
\end{equation}
with $C>0$ only depending on $K$, $L$ and $\max\left\{\frac{R}{\rho_0},1\right\}$.
\end{theorem}


\begin{remark}
It should be stressed here that, in the following sections, we shall not use
the full strength of Theorem~\ref{theo:3spheres_generale} with the
essentially explicit formula \eqref{alpha_generale} for the exponent of the
three-spheres inequality. Actually, we shall merely use the fact that
$\alpha$ only depends on $K$, $L$, $\max\left\{\frac{R}{\rho_0},1\right\}$,
$\frac{r_2}{r_1}$, $\frac{r_3}{r_2}$.We have just recorded here the
representation \eqref{alpha_generale} because it may be of independent
interest. We recall, in particular, that a three-spheres inequality
\eqref{3spheres_generale} with exponent in the form \eqref{alpha_generale}
actually implies the \emph{strong} unique continuation property in a rather
straightforward fashion. The argument is as follows. We fix $r_2, r_3$ and
we allow $r_1 \to 0$, in this case
\begin{equation*}
\alpha=\alpha(r_1)\geq \frac{C_1}{\log\frac{C_2}{r_1}}.
\end{equation*}
With no loss of generality we may assume also $ \|u\|_{L^2(B_{r_3})}\leq 1$.
If for every $N=1,2,\ldots$ we have $ \|u\|_{L^2(B_{r})} = O(r^N)$ as $r
\to 0$, then, for every  $N$, there exists $r'>0$ such that
\begin{equation*} \|u\|_{L^2(B_{r_1})}\leq \left( \frac{r_1}{C_2} \right)^N
, \text{ for every } r_1 < r' ,
\end{equation*}
consequently, by \eqref{3spheres_generale},
\begin{equation*}
\|u\|_{L^2(B_{r_2})}\leq Q \left( \frac{r_1}{C_2}\right)^{\frac{N
C_1}{\log\frac{C_2}{r_1}}} = Q e^{- C_1 N} \to 0, \text{ as } N
\to \infty.
\end{equation*}
See also
Morassi, Rosset and Vessella \cite{morrosves} where this aspect was
investigated in more depth.
\end{remark}

As a first step, we shall prove a three-spheres inequality under the following additional
normalization hypothesis
\begin{equation}
    \label{normalization}
    A(0)=Id.
\end{equation}
Given a weak solution $u\in H^1(B_R)$ to equation \eqref{div_equation}, let us define
\begin{equation}
    \label{mu}
    \mu(x)=\frac{A(x)x\cdot x}{|x|^2},
\end{equation}
\begin{equation}
    \label{H}
    H(r)=\int_{\partial B_r}\mu u^2.
\end{equation}

\begin{theorem}[Three-spheres inequality -- normalized case]
    \label{theo:3spheres}
    If the previously sta\-ted hypo\-the\-ses hold, for
every $r_1$, $r_2$, $r_3$, with $0<r_1<r_2<r_3\leq R$,
\begin{equation}
    \label{3spheres}
    H(r_2)\leq Q H(r_1)^\alpha H(r_3)^{1-\alpha},
\end{equation}
where $Q\geq 1$ only depends on $K$, $L$, $\max\left\{\frac{R}{\rho_0},1\right\}$, and where
\begin{equation}
    \label{alpha}
\alpha=\frac{\log\frac{r_3}{r_2}}{\log\frac{r_3}{r_2}+C\log\frac{r_2}{r_1}},
\end{equation}
with $C>0$ only depending on $K$, $L$ and $\max\left\{\frac{R}{\rho_0},1\right\}$.
\end{theorem}

\begin{remark}
Since $K^{-1}\leq \mu\leq K$, we have
\begin{equation}
    \label{3spheres_bis}
    \int_{\partial B_{r_2}}u^2\leq Q \left(\int_{\partial B_{r_1}}u^2\right)^\alpha
    \left(\int_{\partial B_{r_3}}u^2\right)^{1-\alpha},
\end{equation}
and, by integration and by H\"{o}lder inequality, we also obtain
\begin{equation}
    \label{3balls}
    \int_{ B_{r_2}}u^2\leq Q \left(\int_{ B_{r_1}}u^2\right)^\alpha
    \left(\int_{B_{r_3}}u^2\right)^{1-\alpha}.
\end{equation}
\end{remark}

Let us first derive Theorem~\ref{theo:3spheres_generale}
from Theorem~\ref{theo:3spheres}.

\begin{proof}[Proof of Theorem~\ref{theo:3spheres_generale}]
Let us introduce the change of variables
\begin{equation}
    \label{change_var}
y=Jx,
\end{equation}
where $J=\sqrt{A^{-1}(0)}$ and let us consider, for any $r>0$, the ellipsoid
\begin{equation}
    \label{E_r}
\mathcal{E}_r=\{x\in\R^n \ |\
 (A(0)^{-1} x)\cdot x<r^2\}.
\end{equation}
We have that $\mathcal{E}_r=J^{-1}(B_r)$ and
\begin{equation}
    \label{E_r_B_r}
B_{\frac{r}{\sqrt K}}\subset \mathcal{E}_r\subset B_{\sqrt K r}.
\end{equation}
The function $v(y)=u(J^{-1}y)$ satisfies the elliptic equation
\begin{equation}
    \label{div_equation_v}
    \divrg(\widetilde{A}\nabla v)=0, \quad\hbox{in } B_{\frac{R}{\sqrt K}},
\end{equation}
where $\widetilde{A}(y)=JA(J^{-1}y)J$. It is straightforward to verify that
\begin{equation}
    \label{ellipticity_tilde}
    K^{-2}|\xi|^2\leq
\widetilde{A}(y)\xi\cdot\xi\leq K^2 |\xi|^2,\quad\hbox{for
every }y\in \mathbb{R}^n,\xi\in{\R}^n,
\end{equation}
\begin{equation}
    \label{lipschitz_tilde}
    |\widetilde{A}(y_1)-\widetilde{A}(y_2)|\leq K^{\frac{3}{2}}\frac{L}{\rho_0}|y_1-y_2|,\quad\hbox{for
every }y_1, y_2\in \mathbb{R}^n,\
\end{equation}
\begin{equation}
    \label{normalization_tilde}
    \widetilde{A}(0)=Id.
\end{equation}
Therefore $\widetilde{A}$ satisfies the hypotheses of Theorem~\ref{theo:3spheres}.
Let $\rho_1=\frac{r_1}{\sqrt K}$, $\rho_2=\sqrt K r_2$, $\rho_3=\frac{r_3}{\sqrt K}$. Since
$0<r_1<r_2<\frac{r_3}{K}\leq r_3\leq R$, we have that $0<\rho_1<\rho_2<\rho_3\leq \frac{R}{\sqrt K}$,
and by \eqref{3balls} we have
\begin{equation}
    \label{3balls_v}
    \int_{ B_{\rho_2}}v^2\leq Q \left(\int_{ B_{\rho_1}}v^2\right)^\alpha
    \left(\int_{B_{\rho_3}}v^2\right)^{1-\alpha},
\end{equation}
where $Q\geq 1$ only depends on $K$, $L$, $\max\left\{\frac{R}{\rho_0},1\right\}$,
and where
\begin{equation}
    \label{alpha_balls}
\alpha=\frac{\log\frac{\rho_3}{\rho_2}}{\log\frac{\rho_3}{\rho_2}+C\log\frac{\rho_2}{\rho_1}},
\end{equation}
with $C>0$ only depending on $K$, $L$ and $\max\left\{\frac{R}{\rho_0},1\right\}$.
Coming back to the old variables we have
\begin{equation}
    \label{3ellipsoid}
    \int_{ \mathcal{E}_{\rho_2}}u^2\leq Q \left(\int_{ \mathcal{E}_{\rho_1}}u^2\right)^\alpha
    \left(\int_{\mathcal{E}_{\rho_3}}u^2\right)^{1-\alpha},
\end{equation}
and recalling \eqref{E_r_B_r}, the thesis follows.
\end{proof}

\bigskip

The proof of Theorem~\ref{theo:3spheres}   is essentially based on the following three lemmas.
\begin{lemma}
    \label{lem:H_AC}
The function $H(r)$ is absolutely continuous in $(0,R)$ and, for almost every $r\in (0,R)$,
\begin{equation}
    \label{H'}
    H'(r)=\frac{1}{r}\int_{\partial B_r}(-A\nu\cdot \nu+tr(A)+r\partial_{x_i}a_{ij}\nu_j)u^2
    +2\int_{\partial B_r}uA\nabla u\cdot \nu.
\end{equation}
\end{lemma}

\begin{proof}
We have that
\begin{multline}
    \label{H_AC}
    H(r)=\frac{1}{r}\int_{\partial B_r}u^2A(x)x\cdot\frac{x}{|x|}=
    \frac{1}{r}\int_{\partial B_r}u^2A(x)x\cdot\nu=\\
    =\frac{1}{r}\int_{B_r}\divrg(u^2A(x)x)=\frac{1}{r}\int_{0}^r ds \int_{\partial B_s}\divrg(u^2A(x)x),
\end{multline}
where $\nu$ denotes the outer unit normal to $B_r$.
Since, by well-known regularity results \cite[Theorem~8.32]{gilbargtrud}, $u\in C^{1,\alpha}_{loc}(B_R)$ and since $A\in C^{0,1}(B_R)$, we have that
$\divrg(u^2A(x)x)\in L^\infty(B_R)$ and therefore $H(r)$ is absolutely continuous in $(0,R)$. {}From \eqref{H_AC}, we have
\begin{multline}
    \label{H'_compute}
    H'(r)=-\frac{1}{r}H(r)+\frac{1}{r}\int_{\partial B_r}\divrg(u^2Ax)=\\
    =-\frac{1}{r}\int_{\partial B_r}\mu u^2+\frac{1}{r}\int_{\partial B_r}(tr(A))u^2
    +\frac{1}{r}\int_{\partial B_r}u^2\partial_{x_i}a_{ij}x_j+
    \frac{2}{r}\int_{\partial B_r}uu_{x_i}a_{ij}x_j,
\end{multline}
and, recalling that $A$ is symmetric, \eqref{H'} follows.
\end{proof}

\bigskip

Let us set
\begin{equation}
    \label{I}
    I(r)=\int_{\partial B_r}uA\nabla u\cdot \nu=\int_{B_r}A\nabla u\cdot \nabla u\geq 0.
\end{equation}

\begin{lemma}
    \label{lem:ineq_H'_I'}
There exists a positive constant $C$, only depending on $K$ and $L$, such that, for almost every $r\in (0,R)$,
\begin{equation}
    \label{ineq_H'}
    \left|H'(r)-\frac{n-1}{r}H(r)-2I(r)\right|\leq \frac{C}{\rho_0}H(r),
\end{equation}
\begin{equation}
    \label{ineq_I'}
    I'(r)\geq 2\int_{\partial B_r}\frac{1}{\mu}(A\nabla u\cdot\nu)^2+\frac{n-2}{r}I(r)-\frac{C}{\rho_0}I(r),
\end{equation}
\end{lemma}

\begin{proof}
For almost every $r\in(0,R)$, we can compute
\begin{multline}
    \label{to_ineq H'1}
    H'(r)-\frac{n-1}{r}H(r)-2I(r)=\\
    =\frac{1}{r}\int_{\partial B_r}(-A\nu\cdot \nu+tr(A)+r\partial_{x_i}a_{ij}\nu_j)u^2
    -\frac{n-1}{r}\int_{\partial B_r}\mu u^2=\\
    =\frac{1}{r}\int_{\partial B_r}(-A(0)\nu\cdot \nu+tr(A(0)))u^2
     -\frac{n-1}{r}\int_{\partial B_r}\mu u^2+
    \frac{1}{r}\int_{\partial B_r}r\partial_{x_i}a_{ij}\nu_j u^2+\\
     +\frac{1}{r}\int_{\partial B_r}[-(A-A(0))\nu\cdot \nu+(tr(A)-tr(A(0)))]u^2
    =\\
    =\frac{n-1}{r}\int_{\partial B_r}(1-\mu) u^2+\frac{1}{r}\int_{\partial B_r}r\partial_{x_i}a_{ij}\nu_j u^2+\\
    +\frac{1}{r}\int_{\partial B_r}[-(A-A(0))\nu\cdot \nu+(tr(A)-tr(A(0)))]u^2 .
\end{multline}
By \eqref{lipschitz} and \eqref{normalization}, we have that
\begin{equation}
    \label{A-A(0)}
    |(A(x)-A(0))|\leq \frac{L}{\rho_0}|x|,
\end{equation}
\begin{equation}
    \label{1-mu}
    |1-\mu(x)|=|(A(0)-A(x))\nu\cdot\nu|\leq \frac{L}{\rho_0}|x|,
\end{equation}
and that the absolute values of the three terms in the right-hand side of \eqref{to_ineq H'1}
are bounded by
$CH(r)$, obtaining \eqref{ineq_H'}.

Let us notice that
\begin{equation}
    \label{I_bis}
I(r)=\int_0^rds\int_{\partial B_s}A\nabla u\cdot \nabla u
\end{equation}
is continuously differentiable and that
\begin{equation}
    \label{I'_bis}
    I'(r)=\int_{\partial B_r}A\nabla u\cdot \nabla u.
\end{equation}
Let us recall the following generalization of the Rellich identity \cite{rellich}, due to Payne and
Weinberger \cite{paynewein}
\begin{multline}
    \label{Rellich}
    \int_{\partial B_r}a_{ij}\partial_{x_i}u\partial_{x_j}uf_k\nu_k=\\
    =2\int_{\partial B_r}f_ia_{kj}\partial_{x_i}u\partial_{x_j}u\nu_k-
    2\int_{B_r}f_k\partial_{x_k}u\partial_{x_i}(a_{ij}\partial_{x_j}u)+\\
    +\int_{B_r}(\partial_{x_k}f_k a_{ij}-2\partial_{x_k}f_i a_{kj}+f_k\partial_{x_k}a_{ij})
    \partial_{x_i}u\partial_{x_j}u,
\end{multline}
which holds for every $f\in C^{0,1}(B_r;\R^n)$ and every $u\in H^2(B_r)$, where $a_{ij}$ are as
above.
Choosing
\begin{equation}
    \label{f}
    f(x)=\frac{A(x)x}{r\mu(x)},
\end{equation}
{}from \eqref{A-A(0)} and \eqref{1-mu}, it follows that
\begin{equation}
    \label{deriv_f}
    \left|\partial_{x_i}f_j-\frac{\delta_{ij}}{r}\right|\leq \frac{C}{\rho_0},
\end{equation}
where $C$ is a constant only depending on $L$. Inserting \eqref{f} in \eqref{Rellich} and
recalling \eqref{I'_bis} and \eqref{deriv_f}, \eqref{ineq_I'} easily follows.
\end{proof}

\bigskip

The proof of Theorem~\ref{theo:3spheres} shall be based on a differential inequality
for the so-called \emph{frequency function}, a notion first introduced by Almgren \cite{almgren}, see also \cite{garofalolin1}.
\begin{equation}
    \label{N}
    N(r)=\frac{rI(r)}{H(r)},
\end{equation}
which is well-defined provided $H(r)>0$.
Observe that the thesis of Theorem~\ref{theo:3spheres} is trivial if $u$ is identically
constant in $B_R$, or, as is the same, if $I(R)=0$.
Thus we may assume, with no loss of generality, that $I(R)>0$.
Let us denote
 \begin{equation}\label{rprime}
r'= \inf \{r\in(0,R)\ |\ I(r)>0 \} ,
\end{equation}
and observe that, being $I(r)$ increasing, we also have
 \begin{equation}
  \label{Inull}
I(r)=0 \text{ for every } r\leq r' ,
\end{equation}
\begin{equation}
   \label{Ipos}
I(r)>0 \text{ for every } r, r'<r\leq R,
\end{equation}
\begin{equation}
   \label{u_const}
u(x)\equiv M, \text{ for every } x\hbox{ s.t. } |x|\leq r',
\end{equation}
for some $M\in \R$.

Note that, by the maximum principle, if $H(\rho)=0$ for some $\rho\in (0,R)$, then $u\equiv 0$ in $B_\rho$ and
$H(r)=0$ for every $r<\rho$. Therefore $H(r)>0$ and $N(r)$ is well-defined for every $r\in (r',R]$.

\begin{lemma}\label{lem:Ndiffineq}
There exists a positive constant $C$, only depending on $K$ and $L$, such that $e^{C\frac{r}{\rho_0}}N(r)$ is an increasing function on $(r',R)$.
\end{lemma}
\begin{proof}Let us recall that $I(r)$ is continuously differentiable in $(0,R)$ and that $H(r)$ is absolutely continuous in $(0,R)$, hence
 $N$ is absolutely continuous in $(r',R)$.
 If $r \in (r',R)$, recalling Lemma~\ref{lem:ineq_H'_I'} and using Schwarz inequality, we compute
\begin{multline}
    \label{N'}
    \frac{N'(r)}{N(r)}=\frac{1}{r}+\frac{I'(r)}{I(r)}-\frac{H'(r)}{H(r)}\geq\\
    \geq2\frac{\int_{\partial B_r}\frac{1}{\mu}(A\nabla u\cdot\nu)^2}{\int_{\partial B_r}u(A\nabla u\cdot\nu)}
    -2\frac{\int_{\partial B_r}u(A\nabla u\cdot\nu)}{\int_{\partial B_r}\mu u^2}-\frac{C}{\rho_0}\geq -\frac{C}{\rho_0},
\end{multline}
where $C>0$ only depends on $K$ and $L$.
Therefore
\begin{equation}
    \label{N'+CN}
    N'(r)+\frac{C}{\rho_0}N(r)\geq 0, \quad \hbox{for every } r\in (r',R) ,
\end{equation}
or, equivalently,
\begin{equation}
    \label{e^Cr}
    \frac{d}{dr}(e^{C\frac{r}{\rho_0}}N(r))\geq 0, \quad \hbox{for every } r\in (r',R),
\end{equation}
which proves the lemma.
\end{proof}

\bigskip

\begin{proof}[Proof of Theorem~\ref{theo:3spheres}]
Let $r_1$, $r_2$, $r_3$ be such that $r'<r_1<r_2<r_3\leq R$. By inequality \eqref{ineq_H'} and by
Lemma~\ref{lem:Ndiffineq}, we have
\begin{multline}
    \label{H(r_2)/H(r_1)}
    \log\frac{H(r_2)}{H(r_1)}=\int_{r_1}^{r_2}\frac{H'(r)}{H(r)}dr\leq (n-1)\log\frac{r_2}{r_1}+
    2\int_{r_1}^{r_2}\frac{N(r)}{r}dr+\frac{C}{\rho_0}(r_2-r_1)\leq\\
    \leq (n-1)\log\frac{r_2}{r_1}+e^{\frac{C}{\rho_0}(r_2-r_1)} N(r_2)\log\frac{r_2}{r_1}+\frac{C}{\rho_0}(r_2-r_1),
\end{multline}
\begin{multline}
    \label{H(r_3)/H(r_2)}
    \log\frac{H(r_3)}{H(r_2)}=\int_{r_2}^{r_3}\frac{H'(r)}{H(r)}dr\geq (n-1)\log\frac{r_3}{r_2}+
    2\int_{r_2}^{r_3}\frac{N(r)}{r}dr-\frac{C}{\rho_0}(r_3-r_2)\geq\\
    \geq (n-1)\log\frac{r_3}{r_2}+e^{\frac{C}{\rho_0}(r_2-r_3)} N(r_2)\log\frac{r_3}{r_2}-\frac{C}{\rho_0}(r_3-r_2),
\end{multline}
where $C>0$ denotes a constant, only depending on $K$ and $L$.
{}From \eqref{H(r_2)/H(r_1)} and \eqref{H(r_3)/H(r_2)}, we have
\begin{multline}
    \label{log(H/H)}
    \log\frac{H(r_2)}{H(r_1)}\leq(n-1)\log\frac{r_2}{r_1}+\\
    +e^{\frac{C}{\rho_0}(r_3-r_1)}\left(\frac{\log\frac{H(r_3)}{H(r_2)}}{\log \frac{r_3}{r_2}}-(n-1)
    +\frac{C(r_3-r_2)}{\rho_0\log \frac{r_3}{r_2}} \right)\log\frac{r_2}{r_1}+\frac{C}{\rho_0}(r_2-r_1)\leq\\
    \leq e^{\frac{C}{\rho_0}(r_3-r_1)}\frac{\log \frac{r_2}{r_1}}{\log \frac{r_3}{r_2}}\log\frac{H(r_3)}{H(r_2)}+\frac{C}{\rho_0}(r_3-r_2)e^{C(r_3-r_1)}\frac{\log \frac{r_2}{r_1}}{\log \frac{r_3}{r_2}}+\frac{C}{\rho_0}(r_2-r_1)\leq\\
    \leq p\log\frac{H(r_3)}{H(r_2)}+B,
\end{multline}
where
\begin{equation}
    \label{p}
    p=e^{C\max\left\{\frac{R}{\rho_0},1\right\}} \frac{\log \frac{r_2}{r_1}}{\log \frac{r_3}{r_2}},
\end{equation}

\begin{equation}
    \label{B}
    B=C\max\left\{\frac{R}{\rho_0},1\right\}(1+p)
\end{equation}
and
$C>0$ only depends on $K$ and $L$.
We have
\begin{equation}
    \label{(H/H)}
    \frac{H(r_2)}{H(r_1)}\leq\left( \frac{H(r_3)}{H(r_2)}\right)^pe^B,
\end{equation}
\begin{equation}
    \label{(H/H)bis}
    (H(r_2))^{1+p}\leq e^BH(r_1)(H(r_3))^{p},
\end{equation}
\begin{equation}
    \label{(H/H)ter}
    H(r_2)\leq e^{C\max\left\{\frac{R}{\rho_0},1\right\}}(H(r_1))^{\frac{1}{1+p}}(H(r_3))^{\frac{p}{1+p}}.
\end{equation}
Therefore, assuming the additional condition $r_1>r'$, inequality \eqref{3spheres} follows  with $Q=\exp\left(C\max\left\{\frac{R}{\rho_0},1\right\}\right)$ and
$\alpha=\frac{1}{1+p}$.

It only remains to prove that $r'=0$. Let us assume by
contradiction that $r'>0$. If we had $M=0$ in \eqref{u_const},
then $H(r')=0$ and, passing to the limit in \eqref{3spheres} as
$r_1\to r'$, it would follow  $H(r)=0$ for every $r\in
(r',R)$ and $u\equiv 0$ in $B_R$, leading to a contradiction with
the definition \eqref{rprime} of $r'$. Therefore in this case
$r'=0$. If $M\neq 0$ in \eqref{u_const}, let us consider the
function $v=u-M$, which satisfies $\nabla v \equiv \nabla u$ in
$B_R$ and $v\equiv 0$ in $\overline{B_{r'}}$. By applying the
above arguments to $v$, we have again that $r'=0$ and the proof is
complete.
\end{proof}

\bigskip

\begin{remark}
It is worth emphasizing once more the strength of the frequency function
method just employed. In fact, besides the three-spheres inequality, it
enables also to obtain a doubling inequality, and this was in fact the
original purpose when the method was devised by Garofalo and Lin
\cite{garofalolin1}. Let us outline here how the proof goes. For the sake of
simplicity we stick to the normalized setting of Theorem
\ref{theo:3spheres}. From the proof of Theorem~\ref{theo:3spheres} we have
obtained that $r'=0$ and,  consequently, Lemma~\ref{lem:Ndiffineq} applies
for all $r\in(0,R)$. Hence, by \eqref{ineq_H'} and by Lemma
\ref{lem:Ndiffineq}
\begin{multline}
    \label{H(2r)/H(r)}
    \log\frac{H(2r)}{H(r)}=\int_{r}^{2r}\frac{H'(r)}{H(r)}dr\leq (n-1)\log2+
    2\int_{r}^{2r}\frac{N(s)}{s}ds+\frac{Cr}{\rho_0}\leq\\
    \leq (n-1)\log2+2\int_{r}^{2r}e^{\frac{C}{\rho_0}(R-s)}\frac{N(R)}{s}ds
+\frac{Cr}{\rho_0}\leq\\
    \leq (n-1)\log2 + 2CN(R)\log2+\frac{CR}{\rho_0}.
\end{multline}
That is
\begin{equation} \label{Hdoubling}
H(2r) \leq Q H(r) , \text{ for every  } 0<r\leq \frac{R}{2} ,
\end{equation}
where $Q\geq 1$ only depends on $K,L,\frac{R}{\rho_0}$ and on $N(R)$.
Consequently, by an integration and by the ellipticity bounds
\eqref{ellipticity},
\begin{equation} \label{doublingL2}
\int_{B_{2r}}u^2 \leq Q \int_{B_{r}}u^2  , \text{ for every  } 0<r\leq
\frac{R}{2} ,
\end{equation}
with a possibly different $Q$, but still depending on the same quantities.
As is well-known, \eqref{doublingL2} in turn implies the strong unique
continuation property, because, by iteration on the radii $r_n=R2^{-n}$,
$n=1,2,\ldots$,  we readily arrive at
\begin{equation} \label{polynomialLB}
\int_{B_{r}}u^2 \geq \frac{1}{Q}(\frac{r}{R})^{\frac{\log Q}{\log 2}}
\int_{B_{R}}u^2  , \text{ for every  } 0<r\leq \frac{R}{2} .
\end{equation}
\end{remark}

\section{The two-dimensional case}\label{2Dsec}

In this section we shall deal with the two-dimensional case. We shall prove, by complex-analytic techniques, a three-spheres inequality for an elliptic equation
in divergence form in pure principal part, see Theorem~\ref{threespheresharmteoL2}.
The main difference with respect to higher dimensions is that in two dimensions we allow
the coefficient matrix $A$ to be discontinuous. Actually here we allow $A$ to be also non-symmetric, however we stress that this hypothesis will be used in this section only.

Throughout this section we shall identify $z=x+\rmi y\in\mathbb{C}$ with points $(x,y)\in\mathbb{R}^2$ and we shall fix
$A=A(z)$, $z\in \mathbb{R}^2$,
a real-valued $2\times 2$ matrix such that
its entries are measurable and it satisfies, for a given constant $K\geq 1$,
the ellipticity condition
\begin{equation}\label{aprioricond}
\begin{array}{l}
A(z)\xi\cdot\xi\geq K^{-1}|\xi|^{2},\\
A^{-1}(z)\xi\cdot\xi\geq K^{-1}|\xi|^{2},
\end{array}\quad\text{for almost every }z\in \mathbb{R}^2,
\text{ for every }
\xi\in\mathbb{R}^{2}.
\end{equation}
The constant $K$ will be referred to as the \emph{ellipticity constant} of
$A$. When $A$ is symmetric this condition coincides with the one given in \eqref{ellipticity}. We also observe that \eqref{aprioricond} implies that
$$|a_{ij}(z)|\leq\Lambda,\quad\text{for every }i,j=1,2\text{ and for almost every }z\in \mathbb{R}^2,
$$
$\Lambda$ depending on $K$ only.

In this section, $\Omega$ will as usual denote a bounded connected open set contained in $\mathbb{R}^2$. Letting
$f:\Omega\to \mathbb{C}$ be a complex-valued function,
we shall make repeated use of the following notation
for complex derivatives
$$f_{\overline{z}}=\frac{1}{2}(f_x+\rmi f_y),\quad
f_{z}=\frac{1}{2}(f_x-\rmi f_y).$$
We denote by $J=\left[\begin{smallmatrix}0&-1\\1&0\end{smallmatrix}
\right]$ the counterclockwise rotation of $\frac{\pi}{2}$.

\begin{proposition}\label{strmfuncprop}
Under the above stated hypotheses,
let $u\in H^1_{loc}(\Omega)$ be a weak solution to the equation
\begin{equation}\label{ellipteq}
\mathrm{div}(A\nabla u)=0,\quad\text{in }\Omega,
\end{equation}
where $\Omega$ is a bounded simply connected open set in $\mathbb{R}^2$.

Then there exists a function $v\in H^1_{loc}(\Omega)$ that satisfies
\begin{equation}\label{strmfuncdef}
\nabla v=JA\nabla u,\quad\text{almost everywhere in }\Omega.
\end{equation}
Moreover, letting $f=u+\rmi v$, we have
\begin{equation}\label{qconfeq}
f_{\overline{z}}=\mu f_z+\nu\overline{f_z},\quad
\text{almost everywhere in }\Omega,
\end{equation}
where $\mu$ and $\nu$ are bounded, measurable,
complex-valued coefficients satisfying
\begin{equation}\label{qconfcoeffeq}
|\mu|+|\nu|\leq k<1,\quad\text{almost everywhere in }\Omega,
\end{equation}
$k$ being a constant depending on $K$ only.

On the other hand, if $f=u+\rmi v\in H^1_{loc}(\Omega,\mathbb{C})$
solves \eqref{qconfeq}-\eqref{qconfcoeffeq}, then there exists a real-valued $2\times 2$ matrix
$A\in L^{\infty}(\Omega)$, satisfying
\eqref{aprioricond} with a constant $K$ depending upon $k$ only, such that
$u$ is a weak solution to \eqref{ellipteq}.
\end{proposition}

\begin{proof} By \eqref{strmfuncdef} the vector $JA\nabla u$ is curl-free in the weak sense, thus $v$ is well-defined as long as $\Omega$ is simply connected, see
\cite[Theorem~2.1]{alessmagn} for details.
Then \eqref{qconfeq} follows from \eqref{strmfuncdef} with
$\mu$, $\nu$ given by
\begin{equation}\label{qconfcoeffdef}
\begin{array}{l}
\mu =\frac{\displaystyle{a_{22}-a_{11}-\rmi
(a_{12}+a_{21})}}{\displaystyle{a_{11}a_{22}-a_{12}a_{21}+
a_{11}+a_{22}+1}},\\
\\
\nu =\frac{\displaystyle{a_{12}a_{21}-a_{11}a_{22}+1+\rmi
(a_{12}-a_{21})}}{\displaystyle{a_{11}a_{22}-a_{12}a_{21}+a_{11}+a_{22}+1}}.
\end{array}
\end{equation}
From these expressions and
\eqref{aprioricond},
one obtains, through elementary although lengthy computations,
\eqref{qconfcoeffeq}. In particular, $k$ is related to $K$ by the following formula
$$k\leq \frac{K+\sqrt{K^2-1}-1}{K+\sqrt{K^2-1}+1}.$$
This estimate is sharp and it is proved in \cite{alessnesi}.

Conversely, given the coefficients $\mu$, $\nu$
in \eqref{qconfeq} satisfying \eqref{qconfcoeffeq} one obtains
\eqref{ellipteq} and \eqref{strmfuncdef} with $A$ given by
\begin{equation}\label{Adef}
A=\begin{bmatrix}\displaystyle{\frac{|1-\mu|^2-|\nu|^2}{|1+\nu|^2-|\mu|^2}}
&\displaystyle{\frac{2\Im(\nu-\mu)}{|1+\nu|^2-|\mu|^2}}\\
&\\
\displaystyle{\frac{-2\Im(\mu+\nu)}{|1+\nu|^2-|\mu|^2}}
&\displaystyle{\frac{|1+\mu|^2-|\nu|^2}{|1+\nu|^2-|\mu|^2}}\end{bmatrix}
\end{equation}
and the conclusion follows. It may be shown that the ellipticity constant of $A$ is related to $k$ by the following formula
$$K\leq \frac{1+k}{1-k}.$$
Also this estimate is sharp, see again \cite{alessnesi}.\end{proof}

\bigskip

The function $v$ appearing above is usually called the
\emph{stream function} associated to $u$. Notice that
$v$ is uniquely determined up to an additive constant and also that
$v$ is a weak solution to
\begin{equation}\label{ellipteqstrmfunc}
\mathrm{div}(B\nabla v)=0,\quad\text{in }\Omega,
\end{equation}
where $B=(\det A)^{-1}A^T$, and where $\det (\cdot)$ denotes the determinant and
$(\cdot)^T$ denotes the transpose. We notice that $B$ still satisfies the ellipticity condition \eqref{aprioricond}, with the same ellipticity constant.

For any $k$, $0\leq k<1$, we say that a function $f:\Omega\to \mathbb{C}$ is a
$k$-\emph{quasiconformal function} in $\Omega$
if it is an $H^1_{loc}(\Omega,\mathbb{C})$-solution to the Beltrami-type equation
\eqref{qconfeq}-\eqref{qconfcoeffeq}. A univalent
solution to \eqref{qconfeq}-\eqref{qconfcoeffeq} is said a
$k$-\emph{quasiconformal mapping}. A function $f$ is a
\emph{quasiconformal function}, respectively \emph{mapping},
if it is a $k$-\emph{quasiconformal function}, respectively \emph{mapping},
for some $k$, $0\leq k<1$. Concerning quasiconformal functions, their
properties and characterizations we refer to \cite{LehtoV}.

Finally we observe that when $A$ is the identity matrix, then $u$ is harmonic and $v$ is simply its harmonic conjugate, thus $f=u+\rmi v$ is a holomorphic, or conformal, function, corresponding to $k=0$. Such a special case deserves our attention, especially when $\Omega=B_1$. We have the following result which links properties of a harmonic function $u$ to properties of its harmonic conjugate.

\begin{proposition}\label{Schwarz+Priv}
Let $u$ be a harmonic function in $B_1$, and let $v$ be its harmonic conjugate normalized in such a way that
$v(0)=0$.

Let us assume that $|u(z)|<E$ for every $|z|<1$. Then we have
\begin{equation}\label{Schwarz}
|v(z)|\leq \frac{2E}{\pi}\log\left(\frac{1+|z|}{1-|z|}\right).
\end{equation}

If $u$ is H\"older continuous on $\overline{B_1}$, with H\"older exponent $\beta$, $0<\beta<1$, then $v$ is also H\"older continuous on $\overline{B_1}$, with H\"older exponent $\beta$ and
$$\|v\|_{C^{\beta}(\overline{B_1})}\leq C_0 \|u\|_{C^{\beta}(\overline{B_1})},$$
$C_0$ depending on $\beta$ only.
\end{proposition}

\begin{proof} If $u$ is continuous up to the boundary of $B_1$, then
$$f(z)=u(z)+\rmi v(z)=\frac{1}{2\pi \rmi}\int_{\partial B_1}\frac{\xi+z}{\xi-z}u(\xi)\frac{d\xi}{\xi},\quad\text{for every }|z|<1.$$
Such a representation is often referred to as Schwarz's formula or as Poisson representation.
The inequality \eqref{Schwarz} follows immediately and it is sharp, take for instance $u=\Im\left(\frac{2}{\pi}\log\left(\frac{1+z}{1-z}\right)\right)$ on $B_1$.
The second result is due to Priwaloff, see \cite{Pri}.\end{proof}

\bigskip

Properties of harmonic functions and their conjugates, or of holomorphic functions,
may be transferred to solutions to elliptic equations in divergence form and their stream functions, or to quasiconformal functions, through the following important representation theorem, a proof of which is due to Bers and Nirenberg \cite{bersni}, see also \cite{BJS}.

\begin{theorem}\label{BersNir}
Let $\Omega$ be a connected open set contained in $B_1$ and let
$f\in H^1_{loc}(\Omega,\mathbb{C})$ solve
\eqref{qconfeq} where $\mu$, $\nu$ satisfy \eqref{qconfcoeffeq}.

Then there exist a $k$-quasiconformal mapping $\chi$ from $B_1$ onto itself
and a holomorphic function $F$ on $\chi(\Omega)$ such that
\begin{equation}\label{represform}
f=F\circ\chi.
\end{equation}

We may choose $\chi$ such that $\chi(0)=0$. Moreover we have that
the function $\chi$ and its
inverse $\chi^{-1}$ satisfy the following H\"older continuity properties
\begin{equation}\label{holderchi}
|\chi(x)-\chi(y)|\leq C_1|x-y|^{\beta},\quad\text{for any }x,y\in B_1
\end{equation}
and
\begin{equation}\label{holderchiinv}
|\chi^{-1}(x)-\chi^{-1}(y)|\leq C_1|x-y|^{\beta},\quad
\text{for any }x,y\in B_1,
\end{equation}
where $C_1$ and $\beta$, $0<\beta<1$, depend upon $k$ only.
\end{theorem}

\begin{proof} See \cite[page~116]{bersni} and \cite{BJS}.\end{proof}

\bigskip

Through the use of quasiconformal functions it is also possible to extend
the classical method of harmonic measure
in order to obtain a H\"older stability estimate in the interior for
Cauchy problems for solutions to Beltrami-type equations \eqref{qconfeq}-\eqref{qconfcoeffeq}.
In turn, this provides a H\"older stability estimate in the interior for
Cauchy problems for solutions to \eqref{ellipteq}.

We denote by $\mathcal{L}_A$ the differential operator
\begin{equation}
\mathcal{L}_Au=\mathrm{div}(A\nabla u).
\end{equation}
We now introduce the notion of $\mathcal{L}_A$-harmonic measure $\omega$
of a subset $\Sigma$ of $\partial\Omega$. Roughly speaking $\omega$ is the
solution to the Dirichlet problem
\begin{equation}\label{harmmeas}
\left\{\begin{array}{ll}
\mathrm{div}(A\nabla \omega)=0, &\text{in }\Omega,\\
\omega=1, &\text{on }\Sigma,\\
\omega=0, &\text{on }\partial\Omega\setminus \Sigma.
\end{array}\right.
\end{equation}
However, being the set $\Sigma$ arbitrary, and since the Dirichlet data in
\eqref{harmmeas} may be discontinuous, a more careful definition is needed.
Let us recall some notions from potential theory, see for instance
\cite{heinonen}.

\begin{defin}
A function $u:\Omega\mapsto\mathbb{R}\cup\{+\infty\}$ is called
$\mathcal{L}_A$-\emph{superharmonic} in $\Omega$ if
\begin{enumerate}[\textnormal{(}i\textnormal{)}]
\item $u$ is lower semicontinuous\textnormal{;}
\item $u\not\equiv+\infty$ in any connected component of $\Omega$\textnormal{;}
\item for any open set $\Omega_1\subset\subset \Omega$ and any
$h\in C(\overline{\Omega_1})$, such that $\mathcal{L}_Ah=0$ in the weak sense in
$\Omega_1$, if $u\geq h$ on $\partial \Omega_1$ then $u\geq h$ in $\Omega_1$.
\end{enumerate}
A function $u$ is $\mathcal{L}_A$-\emph{subharmonic} in $\Omega$ if $-u$
is $\mathcal{L}_A$-superharmonic in $\Omega$.
\end{defin}

\begin{defin}\label{hrmmeasdef}
Let $\Sigma$ be a subset of $\partial \Omega$
and let $\chi_\Sigma$ be its characteristic function.
We define
$\mathcal{U}_\Sigma$ as the class of
the $\mathcal{L}_A$-superharmonic functions $u$ in $\Omega$ such that
$u\geq0$ and $\liminf_{x\to y}u(x)\geq\chi_\Sigma(y)$ for any $y\in\partial \Omega$.

We define the $\mathcal{L}_A$-\emph{harmonic measure} of $\Sigma$
with respect to $\Omega$ as the \emph{upper Perron solution} with respect
to $\chi_\Sigma$, that is
$$\omega(z)=\omega(\Sigma,\Omega,\mathcal{L}_A;z)=\inf\{u(z)\ |\ u\in\mathcal{U}_\Sigma\},
\quad\text{for any }z\in \Omega.$$
\end{defin}

We observe that we have $0\leq \omega\leq 1$ everywhere. In order to have that
$0<\omega<1$ in $\Omega$, we need to guarantee that $\omega$ is not identically equal to $0$ or to $1$. A simple sufficient condition is the following, \cite[Theorem~11.6]{heinonen}. We recall that a \emph{continuum} is a closed connected set with at least two points.
There exist two continua, $\gamma_1$ and $\gamma_2$,
such that $\gamma_1$ is contained in $\Sigma$ and has positive distance from
$\partial\Omega\setminus\Sigma$ and
$\gamma_2$ is contained in $\partial \Omega\setminus\Sigma$ and has positive distance from $\Sigma$. For a more thorough analysis of this question we refer
to \cite{heinonen}.

\begin{lemma}\label{subharmlem}
Let $f\in H^1_{loc}(\Omega,\mathbb{C})$ be a non-identically zero solution to
\eqref{qconfeq}-\eqref{qconfcoeffeq}. Then there exists a real-valued $2\times 2$ matrix
$A_1\in L^{\infty}(\Omega)$
such that $A_1$
satisfies \eqref{aprioricond}
with a constant $K$ depending on $k$ only, and
$\varphi=\log|f|$ is $\mathcal{L}_{A_1}$-subharmonic.
\end{lemma}

\begin{proof}
We denote
$$\mu_1=\left\{\begin{array}{ll}
\mu, &\text{where }f_z=0,\\
\mu+\nu\overline{f_z}/f_z,&\text{where }f_z\neq 0.
\end{array}\right.$$
Then we have that $|\mu_1|\leq k$ almost everywhere in $\Omega$ and
$f$ satisfies
$$f_{\overline{z}}=\mu_1 f_z,\quad
\text{almost everywhere in }\Omega.$$

Let $z$ be a point in $\Omega$
such that $f(z)\neq 0$. Locally, on a neighbourhood
of $z$, we can define the function $\Phi=\log f$
where $\log$ is any possible determination of the logarithm
in the complex plane.
In this neighbourhood $\Phi$ also satisfies
\begin{equation}
\Phi_{\overline{z}}=\mu_1 \Phi_z.
\end{equation}

Then by Proposition~\ref{strmfuncprop}
the function $\varphi=\log|f|=\Re\log f$ locally satisfies
\begin{equation}\label{relogflocaleq}
\mathrm{div}(A_1\nabla\varphi)=0
\end{equation}
where the matrix $A_1$ is given by \eqref{Adef} with $\mu$ and $\nu$ replaced by $\mu_1$ and $0$, respectively. Note incidentally that $A_1$ is symmetric, that is $A_1=A_1^T$,
$\det A_1=1$, and hence $B_1=(\det A_1)^{-1}A_1^T=A_1$.

We remark that we can define $\varphi=\log|f|$ globally
as an $H^1_{loc}(\Omega_1)$ function, where $\Omega_1=\{z\in \Omega\ |\ f(z)\neq 0\}$,
hence using a partition of unity it is easy to show that
\eqref{relogflocaleq} holds weakly in $\Omega_1$.

Clearly, by Theorem~\ref{BersNir}, the set
$\{z\in \Omega\ |\ f(z)=0\}$ consists of isolated points and $\varphi$ goes uniformly
to $-\infty$ as $z$ converges to an element of such a set.

Using this remark and the maximum principle we can prove in an elementary
way that
$\varphi=\log|f|$ is $\mathcal{L}_{A_1}$-subharmonic.\end{proof}

\bigskip

We observe that if $k=0$, that is $f$ is holomorphic, then $A_1$ in the
previous lemma is the identity matrix.
By the use of the $\mathcal{L}_{A_1}$-harmonic measure
we obtain a H\"older stability estimate in the interior for Cauchy problems
for a Beltrami-type equation \eqref{qconfeq}-\eqref{qconfcoeffeq}.

\begin{theorem}\label{harmmeastechthm}
Let $\Sigma$ be a subset of $\partial \Omega$.
Let $f\in H^1_{loc}(\Omega,\mathbb{C})$ solve
\eqref{qconfeq}-\eqref{qconfcoeffeq}.

If $\|f\|_{L^{\infty}(\Omega)}\leq E$ and we have that, given $\eta>0$,
\begin{equation}\label{Cauchydataerror}
\limsup_{x\to y}|f(x)|\leq\eta,\quad\text{for any }y\in \Sigma,
\end{equation}
then for any $z\in \Omega$ the following estimate holds
\begin{equation}\label{intCauchyest}
|f(z)|\leq E^{1-\omega(z)}\eta^{\omega(z)},
\end{equation}
where $\omega=\omega(\Sigma,\Omega,\mathcal{L}_{A_1})$ is the
$\mathcal{L}_{A_1}$-harmonic measure of $\Sigma$ with respect to $\Omega$
and the matrix $A_1$ is defined as in Lemma~\ref{subharmlem}.
\end{theorem}

\begin{proof} This result was obtained in \cite[Theorem~4.5]{alerondiSIAM}.
We sketch a proof, for the sake of completeness.
It is evident that $0<\eta\leq E$. If $\eta=E$ the proof is trivial.
We consider then the case $0<\eta<E$.
By Lemma~\ref{subharmlem}
the function $\varphi=\log|f|$ is $\mathcal{L}_{A_1}$-subharmonic.
Let $\omega=\omega(\Sigma,\Omega,\mathcal{L}_{A_1})$ be the
$\mathcal{L}_{A_1}$-harmonic measure of $\Sigma$ with respect to $\Omega$.

Let us denote
$$\widetilde{\varphi}=\frac{\log \Big(\frac{|f|}{E}\Big)}{\log\Big(\frac{\eta}{E}\Big)}.$$
It is easy to see that $\widetilde{\varphi}$ belongs to the upper class
$\mathcal{U}_\Sigma$. Hence for any $z\in \Omega$ we have
$\omega(z)\leq\widetilde{\varphi}(z)$ and consequently
\begin{equation}
\varphi(z)\leq\log(\eta)\omega(z)+\log(E)(1-\omega(z)).
\end{equation}
Thus the conclusion follows.\end{proof}

\bigskip

We observe that, in view of Proposition~\ref{strmfuncprop},
the above Theorem~\ref{harmmeastechthm}
could be restated in terms of a Cauchy problem for an elliptic equation
like \eqref{ellipteq}.

\begin{theorem}\label{harmmeastechthm2}
Let $\Sigma$ be a subset of $\partial \Omega$.
Let $u\in H^1(\Omega)$ be a solution to \eqref{ellipteq}.
Let us assume that there exists
a single-valued stream function $v$ in $\Omega$ and let $f=u+\rmi v$.
We a-priori assume that $\|f\|_{L^{\infty}(\Omega)}\leq E$.

We assume that $\Sigma\subset\partial \Omega$ is a connected open Lipschitz arc of $\partial \Omega$ and as usual we denote $g=u|_{\Sigma}$ and $\psi=A\nabla u\cdot\nu|_{\Sigma}$ the Cauchy data of $u$ on $\Sigma$. We suppose that $g\in C^0(\Sigma)$ and
$\psi\in L^1(\Sigma)$.

If, given $\eta>0$,
$$\|g\|_{L^{\infty}(\Sigma)}+\|\psi\|_{L^1(\Sigma)}\leq \eta,$$
then
for any $z\in \Omega$ the following estimate holds
\begin{equation}\label{intCauchyest2}
|u(z)|\leq E^{1-\omega(z)}\eta^{\omega(z)},
\end{equation}
where $\omega=\omega(\Sigma,\Omega,\mathcal{L}_{A_1})$ is the
$\mathcal{L}_{A_1}$-harmonic measure of $\Sigma$ with respect to $\Omega$
and the matrix $A_1$ is defined as in Lemma~\ref{subharmlem} with
$\mu$ and $\nu$ given by \eqref{qconfcoeffdef}.
\end{theorem}

\begin{proof}
By our assumptions on $g$ and $\psi$, and classical regularity estimates, we have that $f$ is continuous at any point $y\in\Sigma$.

We normalize $v$ in such a way that $v=0$ on some point of $\Sigma$, and we infer that
$\|f\|_{L^{\infty}(\Sigma)}\leq \eta$ and, in view of the regularity of $f$, we obtain that
\eqref{Cauchydataerror} holds. Therefore the conclusion is an immediate consequence of Theorem~\ref{harmmeastechthm}.
\end{proof}

\bigskip

It is evident that in order to make such estimates practically useful, it is necessary to provide a positive lower bound on
$\omega$ and this is usually obtained by using the Harnack inequality for positive solutions to elliptic equations in divergence form, \cite[Theorem~8.20]{gilbargtrud}, see for details \cite{alerondiSIAM}.

Our principal aim here is however to concentrate on three-spheres inequalities, first for holomorphic and quasiconformal functions, then for harmonic functions and solutions to elliptic equations in divergence form.

\begin{theorem}\label{threespheresconfteo}
Let $R>0$ and $0\leq k<1$, and let $f$ be a $k$-quasiconformal function on $B_R$.
Then for every $r_1$, $r_2$, $r_3$, with $0<r_1<r_2<r_3\leq R$,
there exists a constant $\alpha$, $0<\alpha<1$, depending on $k$, $\frac{r_2}{r_1}$ and $\frac{r_3}{r_2}$ only,
such that
\begin{equation}\label{3spheresquasiconf}
\|f\|_{L^{\infty}(B_{r_2})}\leq \|f\|^{\alpha}_{L^{\infty}(B_{r_1})}\|f\|^{1-\alpha}_{L^{\infty}(B_{r_3})}.
\end{equation}
\end{theorem}

\begin{proof} We may use the previous
Theorem~\ref{harmmeastechthm}, by setting $\Omega=B_{r_3}\setminus \overline{B_{r_1}}$ and $\Sigma=\partial B_{r_1}\subset \partial \Omega$. Then
$$\alpha=\inf_{z\in B_{r_2}\setminus \overline{B_{r_1}}}\omega(\Sigma,\Omega,\mathcal{L}_{A_1})(z).$$

For example, when $k=0$, by this technique we may recover the three-circles theorem for holomorphic functions by Hadamard. In fact, in this case $A_1$ is the identity matrix and by explicit computations in radial coordinates we obtain
$$\alpha=\frac{\log(\frac{r_3}{r_2})}{\log(\frac{r_3}{r_1})}, \quad\quad\text{when }k=0.$$

In order to obtain an essentially optimal value of $\alpha$ in the general case, that is when $0\leq k<1$, we
use the representation theorem, Theorem~\ref{BersNir}.
A related argument was used already in \cite[Proposition~1]{aleesca}.
Up to a linear change of variables, we may temporarily assume that $r_3=1$. We apply Theorem~\ref{BersNir}, then
$f=F\circ\chi$, where $F$ is conformal on $B_1$ and $\chi$ is a quasiconformal mapping between $B_1$ and itself.
Let $C_1$ and $\beta$, $0<\beta<1$, be the constants appearing in
\eqref{holderchi} and \eqref{holderchiinv}, which characterize the H\"older continuity of the quasiconformal mapping $\chi$ and of its inverse. Then there exist $0<\widetilde{r}_1<\widetilde{r}_2<1$ such that
$B_{\widetilde{r}_1}\subset \chi(B_{r_1})$ and
$\chi(B_{r_2})\subset B_{\widetilde{r}_2}$ or, as is the same, $\chi(B_1\setminus\overline{B{r_2}})\supset B_1\setminus\overline{B_{\widetilde{r}_2}}$.
We apply the three-circles theorem to $F$, and we obtain that
\eqref{3spheresquasiconf} holds with
\begin{equation}\label{alphaqconf}
\alpha=\frac{\log(\frac{1}{\widetilde{r}_2})}{\log(\frac{1}{\widetilde{r}_1})},\quad\quad\text{when }0\leq k<1,
\end{equation}
where we may take
\begin{equation}\label{tilder}
\widetilde{r}_1=\left(\frac{r_1}{C_1r_3}\right)^{1/\beta}\quad\text{and}\quad\widetilde{r}_2=1-\left(\frac{1-\frac{r_2}{r_3}}{C_1}\right)^{1/\beta},
\end{equation}
$C_1$ and $\beta$, $0<\beta<1$, depending on $k$ only. The above choices of $\widetilde{r}_1$ and $\widetilde{r}_2$ serve our purposes for any $r_3>0$.
\end{proof}

\bigskip

We now turn our attention to the second order elliptic equation.
We shall prove a three-spheres inequality first for $L^{\infty}$-norms and then for $L^2$-norms.
Let $R>0$ and let $u\in H^1_{loc}(B_R)$ solve
\begin{equation}\label{ellipteq2}
\mathrm{div}(A\nabla u)=0,\quad\text{in }B_R.
\end{equation}

\begin{theorem}\label{threespheresharmteo}
Under the above stated hypotheses,
for every $r_1$, $r_2$, $r_3$, with $0<r_1<r_2<r_3\leq R$,
there exist constants $Q\geq 1$ and $\alpha$, $0<\alpha<1$, depending on $K$, $\frac{r_2}{r_1}$ and $\frac{r_3}{r_2}$ only,
such that
\begin{equation}\label{3spheresellipt}
\|u\|_{L^{\infty}(B_{r_2})}\leq Q\|u\|^{\alpha}_{L^{\infty}(B_{r_1})}\|u\|^{1-\alpha}_{L^{\infty}(B_{r_3})}.
\end{equation}
\end{theorem}

\begin{proof}
We temporarily assume that $r_3=1$ and also $A=Id$, that is $u$ harmonic in $B_1$.
We let $v$ be its harmonic conjugate normalized in such a way that $v(0)=0$.
By \eqref{Schwarz}, we may bound the values of $v$ in terms of those on $u$. For example, we have
$$\|v\|_{L^{\infty}(B_{\frac{r_1}{2}})}\leq C_2\|u\|_{L^{\infty}(B_{r_1})}\quad\text{and}
\quad \|v\|_{L^{\infty}(B_{\frac{r_2+1}{2}})}\leq C_3\|u\|_{L^{\infty}(B_1)}$$
with
\begin{equation}\label{threecirclescoeff1}
C_2=\frac{2}{\pi}\log 3
\quad\text{and}\quad C_3=\frac{2}{\pi}\log\left(\frac{3+\frac{r_2}{r_3}}{1-\frac{r_2}{r_3}}\right).
\end{equation}
Therefore, we may use the Hadamard three-circles theorem and conclude that
for $u$ harmonic in $B_R$ we have
\begin{equation}\label{threecircles}
\|u\|_{L^{\infty}(B_{r_2})}\leq ((C_3+1)\|u\|_{L^{\infty}(B_{r_3})})^{1-\alpha}((C_2+1)\|u\|_{L^{\infty}(B_{r_1})})^{\alpha}
\end{equation}
with $\alpha$ given by
\begin{equation}\label{threecirclescoeff2}
\alpha=\frac{\log\big(\frac{1}{2}+\frac{r_3}{2r_2}\big)}{\log\big((1+\frac{r_2}{r_3})\frac{r_3}{r_1}\big)}.
\end{equation}

For the general case of solutions to elliptic equations in divergence form, we use the
same argument as in Theorem~\ref{threespheresconfteo}.
By Proposition~\ref{strmfuncprop} and Theorem~\ref{BersNir},
we have that $u=U\circ\chi$ where $U$ is a harmonic function in  $B_1$
and $\chi$ is a quasiconformal mapping between $B_1$ and itself.
Then we apply \eqref{threecircles} to $U$ and we conclude that
\eqref{3spheresellipt} is satisfied in the following form. It holds
\eqref{threecircles} under conditions
\eqref{threecirclescoeff1}, \eqref{threecirclescoeff2}, with $\frac{r_1}{r_3}$ and $\frac{r_2}{r_3}$ replaced by $\widetilde{r}_1$ and $\widetilde{r}_2$, respectively,
$\widetilde{r}_1$ and $\widetilde{r}_2$ given by \eqref{tilder}.\end{proof}

\bigskip

We remark that for harmonic functions, instead of \eqref{threecircles}, we might have used a result by Korevaar and Meyers, \cite{kormey}. They proved
that for $u$ harmonic in $B_R$ a three-spheres inequality holds with $Q=1$, that is
\begin{equation}\label{KorMey}
\|u\|_{L^{\infty}(B_{r_2})}\leq \|u\|_{L^{\infty}(B_{r_3})}^{1-\widetilde{\alpha}}\|u\|_{L^{\infty}(B_{r_1})}^{\widetilde{\alpha}},
\end{equation}
where
\begin{equation}\label{KorMeycoeff}
\widetilde{\alpha}=\widetilde{\alpha}\left(\frac{r_2}{r_1},\frac{r_3}{r_2}\right).
\end{equation}
We observe that their result is actually valid in any dimension $n\geq 2$.
Let us also recall an interesting related result by Petrosyan \cite{petro}.

As a corollary to Theorem~\ref{threespheresharmteoL2} we obtain the corresponding three-spheres inequality in $L^2$-norms as well.
\begin{theorem}[Three-spheres inequality -- two-dimensional case]\label{threespheresharmteoL2}
Under the previously stated hypotheses,
for every $r_1$, $r_2$, $r_3$, with $0<r_1<r_2<r_3\leq R$,
there exist constants $Q\geq 1$ and $\alpha$, $0<\alpha<1$, depending on $K$, $\frac{r_2}{r_1}$ and $\frac{r_3}{r_2}$ only,
such that
\begin{equation}\label{3sphereselliptL2}
\|u\|_{L^{2}(B_{r_2})}\leq Q\|u\|^{\alpha}_{L^{2}(B_{r_1})}\|u\|^{1-\alpha}_{L^{2}(B_{r_3})}.
\end{equation}
\end{theorem}

\begin{proof}
We use the local boundedness estimate for weak solutions to \eqref{ellipteq2},
see for instance \cite[Theorem~8.17]{gilbargtrud}, telling that for a constant $\widetilde{C}$ depending on $K$ only, and for every $s$, $r$, $0<s<r\leq R$, we have
$$\|u\|_{L^{\infty}(B_s)}\leq \frac{\widetilde{C}\rho_0}{r-s}\|u\|_{L^{2}(B_r)}.$$
We apply this bound with the following choices of $s$, $r$
$$s=\frac{r_2+r_3}{2},\quad r=r_3,$$
and also
$$s=\frac{r_1}{2},\quad r=r_1,$$
in combination with Theorem~\ref{threespheresharmteo} with
radii $0<\frac{r_1}{2}<r_2<\frac{r_2+r_3}{2}$.
Hence, by the trivial estimate
$$\|u\|_{L^{2}(B_{r_2})}\leq \frac{\sqrt{\pi}r_2}{\rho_0}\|u\|_{L^{\infty}(B_{r_2})},$$
we conclude that
$$\|u\|_{L^{2}(B_{r_2})}\leq \sqrt{\pi}Q\widetilde{C}
\left(\frac{4r_2}{r_3-r_2}\|u\|_{L^{2}(B_{r_3})}\right)^{1-\alpha}
\left(\frac{2r_2}{r_1}\|u\|_{L^{2}(B_{r_1})}\right)^{\alpha}$$
where $Q$ and $\alpha$ are the constants appearing in \eqref{3spheresellipt}
related to $K$, $\frac{2r_2}{r_1}$ and $\frac{r_2+r_3}{2r_2}$.\end{proof}


\section{The three-spheres inequality for the complete equation}\label{3spheresgensec}

Here and in the rest of the paper we consider the elliptic operator
$\mathrm{div}(A\nabla\cdot)+c\,\cdot$ where $A$ and $c$ satisfy the conditions
\eqref{ellipticity}, \eqref{lipschitz},
and \eqref{c_bound}, respectively,
stated in the Introduction. We recall that the Lipschitz condition \eqref{lipschitz} is invoked only when $n\geq 3$.

We start by considering solutions $u$ to the homogeneous equation
\begin{equation}
    \label{div_equation_zero_order}
    \divrg(A\nabla u)+cu=0, \quad\hbox{in } B_R.
\end{equation}

\begin{theorem}[Three-spheres inequality -- equation with zero order term]
    \label{cor:3spheres_zero_order}
    If the abo\-ve sta\-ted hypotheses hold, there exists $C_0$, $0<C_0\leq 1$, only depending on
$K$, $L$ and $\kappa$, such that, setting
\begin{equation}
    \label{R_0}
R_0=\min\{R,C_0\rho_0\},
\end{equation}
for
every $r_1$, $r_2$, $r_3$, with $0<r_1<r_2<\frac{r_3}{4K}\leq r_3\leq R_0$,
\begin{equation}
    \label{3spheres_zero_order}
    \|u\|_{L^2(B_{r_2})}\leq Q \|u\|_{L^2(B_{r_1})}^\alpha \|u\|_{L^2(B_{r_3})}^{1-\alpha},
\end{equation}
where $Q\geq 1$ and $\alpha$, $0<\alpha<1$ only depend on $K$, $L$, $\kappa$, $\max\left\{\frac{R}{\rho_0},1\right\}$, $\frac{r_2}{r_1}$, $\frac{r_3}{r_2}$.
\end{theorem}
The thesis is an immediate consequence of the following two lemmas.
\begin{lemma}
   \label{lem:3spheres_zero_order1}
For every $\delta>0$ there exists $C_0$, $0<C_0\leq 1$, only depending on $K$, $L$, $\kappa$
and $\delta$ such that, denoting $R_0=\min\{R,C_0\rho_0\}$,
there exists a positive solution $w\in C^1(B_{R_0})$ to
\begin{equation}
    \label{div_equation_zero_order_w}
    \divrg(A\nabla w)+cw=0, \quad\hbox{in } B_{R_0},
\end{equation}
such that
\begin{equation}
    \label{bounds_w}
    \frac{1}{1+\delta^2}\leq w\leq 1+\delta^2, \quad\hbox{in } B_{R_0},
\end{equation}
and when $n\geq 3$
\begin{equation}
    \label{bounds_nablaw}
    |\nabla w|\leq \frac{\delta}{\rho_0}, \quad\hbox{in } B_{R_0}.
\end{equation}
\end{lemma}

\begin{proof}
There exists $C_1>0$ only depending on $K$, $\kappa$, such that the operator
$\divrg(A\nabla\ \cdot)+c\ \cdot$ is coercive on $H^1_0(B_r)$ for every $r\leq R_1$,
with $R_1=\min\{R,C_1\rho_0\}$, see for instance \cite[Lemma~8.4]{gilbargtrud}.
Let us pick the unique solution $w$ to
\begin{equation}
  \label{Dir_w}
  \left\{ \begin{array}{ll}
  \divrg (A \nabla w)+cw=0, &
  \mathrm{in}\ B_r ,\\
  & \\
  w=1, &
  \mathrm{on}\ \partial B_r.\\
  \end{array}\right.
\end{equation}
Denoting $z=w-1$, we have
\begin{equation}
  \label{Dir_z}
  \left\{ \begin{array}{ll}
  \divrg (A \nabla z)=-c(1+z), &
  \mathrm{in}\ B_r ,\\
  & \\
  z=0, &
  \mathrm{on}\ \partial B_r.\\
  \end{array}\right.
\end{equation}
By standard a-priori bounds in $L^\infty$, see \cite[Theorem~8.16]{gilbargtrud},
we have
\begin{equation}
  \label{bound_z}
  \|z\|_{L^\infty(B_r)}\leq C\kappa \frac{r^2}{\rho_0^2}(1+\|z\|_{L^\infty(B_r)}),
\end{equation}
where $C$ only depends on $K$.
Hence there exists $C_2\leq C_1$ only depending on $K$, $\kappa$,
such that
for every $r\leq R_2$, with $R_2=\min\{R,C_2\rho_0\}$
we have
\begin{equation}
  \label{bound_zbis}
  \|z\|_{L^\infty(B_r)}\leq C\kappa \frac{r^2}{\rho_0^2}.
\end{equation}
Next, when $n\geq 3$,
applying to \eqref{Dir_z} a global Schauder type estimate, see \cite[Theorem~8.33]{gilbargtrud},
one obtains, for every $r\leq R_2$,
\begin{equation}
  \label{bound_nabla_z}
  \|\nabla z\|_{L^\infty(B_r)}\leq C\kappa \frac{r}{\rho_0^2}(1+\|z\|_{L^\infty(B_r)}),
\end{equation}
where $C$ only depends on $K$, $L$.
Thus, we may find $C_0\leq C_2$ only depending on $K$, $\kappa$, $L$ and $\delta$,
such that, denoting
$R_0=\min\{R,C_0\rho_0\}$,
we have
\begin{equation}
  \label{bound_zter}
  \|z\|_{L^\infty(B_{R_0})}\leq \frac{\delta^2}{1+\delta^2},
\end{equation}
and when $n\geq 3$
\begin{equation}
  \label{bound_nabla_zbis}
  \|\nabla z\|_{L^\infty(B_{R_0})}\leq \frac{\delta}{\rho_0}.
\end{equation}
Therefore \eqref{bounds_w} and \eqref{bounds_nablaw} follow immediately.

Observe that, for the sake of simplicity and with no loss of generality, we can also assume $C_0\leq 1$.
\end{proof}

\bigskip

\begin{lemma}
   \label{lem:3spheres_zero_order2}
Let $R_0$ be the quantity introduced in Lemma~\ref{lem:3spheres_zero_order1}, when $\delta$
is chosen to be $\delta=1$ when $n=2$ and
$\delta=\min\{\frac{L}{K},1\}$ when $n\geq 3$. Then $u$ can be factored in $B_{R_0}$ as
\begin{equation*}
    u=wv,
\end{equation*}
where $w$ is the function constructed in Lemma~\ref{lem:3spheres_zero_order1} and $v$ solves
\begin{equation}
    \label{div_equation_zero_order_v}
    \divrg(\widetilde{A}\nabla v)=0, \quad\hbox{in } B_{R_0},
\end{equation}
where the matrix $\widetilde{A}$ is given by
\begin{equation}
    \label{A_tilde}
    \widetilde{A}=w^2 A,
\end{equation}
and it satisfies
\begin{equation}
    \label{ellipticity_tilde_bis}
    \frac{1}{4K}|\xi|^2\leq\widetilde{A}(x)\xi\cdot\xi\leq 4K|\xi|^2,\quad\hbox{for almost
every }x\in B_{R_0},\text{ for every }\xi\in{\R}^n.
\end{equation}
and when $n\geq 3$
\begin{equation}
    \label{lipschitz_tilde_bis}
    |\widetilde{A}(x)-\widetilde{A}(y)|\leq \frac{8L}{\rho_0}|x-y|,\quad\hbox{for
every }x,y\in B_{R_0}.
\end{equation}
\end{lemma}
\begin{proof}
The proof is straightforward.
\end{proof}

\bigskip

Let us now consider the inhomogeneous elliptic equation
\begin{equation}
    \label{div_equation_fF}
    \divrg(A\nabla u)+cu=f+\divrg F, \quad\hbox{in } B_R,
\end{equation}
where
$f\in L^2(\R^n)$ and $F\in L^2(\R^n; \R^n)$ satisfy \eqref{bound_fF}.

\begin{theorem}[Three-spheres inequality -- complete equation]
    \label{cor:3spheres_nonhomogeneous}
    If the the above sta\-ted hypotheses hold, for
every $r_1$, $r_2$, $r_3$, with $0<r_1<r_2<\frac{r_3}{4K}\leq r_3\leq R_0$,
\begin{equation}
    \label{3spheres_nonhomogeneous}
    \|u\|_{L^2(B_{r_2})}\leq Q \left(\|u\|_{L^2(B_{r_1})}+\varepsilon\right)^\alpha
    \left(\|u\|_{L^2(B_{r_3})}+\varepsilon\right)^{1-\alpha},
\end{equation}
where $Q\geq 1$ and $\alpha$, $0<\alpha<1$, only depend on $K$, $L$, $\kappa$, $\max\left\{\frac{R}{\rho_0},1\right\}$,
$\frac{r_2}{r_1}$, $\frac{r_3}{r_2}$,
and $R_0$ is given by \eqref{R_0}.
\end{theorem}

\begin{proof}
Let us consider the unique solution $u_0$ to
\begin{equation}
  \label{Dir_u_0}
  \left\{ \begin{array}{ll}
  \divrg (A \nabla u_0)+cu_0=f+\divrg F, &
  \mathrm{in}\ B_{R_0} ,\\
  & \\
  u_0=0, &
  \mathrm{on}\ \partial B_{R_0}.\\
  \end{array}\right.
\end{equation}
We have that
\begin{equation}
  \label{bound_u_o}
   \|u_0\|_{L^2(B_{R_0})}\leq C(
   R_0^2\|f\|_{L^2(\R^n)}+R_0\|F\|_{L^2(\R^n;\R^n)})
\end{equation}
with $C$ only depending on $K$. Noticing that $u-u_0$ satisfies the hypotheses of Theorem~\ref{cor:3spheres_zero_order}, the thesis follows immediately.
\end{proof}


\section{Propagation of smallness} \label{sec:
PS}

For every $G\subset \R^n$ and for every $h>0$ we shall denote
\begin{equation}
  \label{eq:int_env}
  G_{h}=\{x \in G\ |\ \mathrm{dist}(x,\partial G)>h
  \},
\end{equation}
\begin{equation}
  \label{eq:est_env}
  G^{h}=\{x \in \R^n\ |\ \mathrm{dist}(x,\overline{G})<h
  \}.
\end{equation}

\begin{theorem}[Propagation of smallness in the interior]
    \label{theo:PSinterior}
    Let $\Omega$ be a bounded connected open set in $\R^n$, and let $B_{r_0}(x_0)\subset\Omega$ be a fixed ball. Let $C_0$ be as in the thesis of
Theorem~\ref{cor:3spheres_zero_order}.
      Let $h$, $0<h\leq \min\{2C_0\rho_0,\frac{r_0}{2}\}$, be fixed
    and let $G\subset\Omega$ be a connected open set such that
    $\mathrm{dist}(G,\partial\Omega)\geq h$ and $B_{\frac{r_0}{2}}(x_0)\subset G$.

 Let $u\in H^1_{loc}(\Omega)$ be a solution to the equation
\begin{equation}
    \label{div_equation_fF_bis}
    \divrg(A\nabla u)+cu=f+\divrg F, \quad\hbox{in } \Omega,
\end{equation}
where
$f$ and $F$ satisfy \eqref{bound_fF}. Let us assume that
\begin{equation}
    \label{sigma_0}
    \|u\|_{L^2(B_{r_0}(x_0))}\leq\eta,
\end{equation}
\begin{equation}
    \label{E_0}
    \|u\|_{L^2(\Omega)}\leq E_0,
\end{equation}
for given $\eta>0$, $E_0>0$. We have
\begin{equation}
    \label{PSinterior}
    \|u\|_{L^2(G)}\leq C\left(\eta+\varepsilon\right)^{\delta}
    (E_0+\varepsilon)^{1-\delta},
\end{equation}
where
\begin{equation}\label{Cdef}
C= C_1\left(\frac{|\Omega|}{h^n}\right)^{\frac{1}{2}}
\end{equation}
and
\begin{equation}\label{deltadef}
\delta\geq \alpha^{\frac{C_2|\Omega|}{h^n}}
\end{equation}
with $C_1>0$ and $\alpha$, $0<\alpha<1$, only depending on $K$, $L$ and $\kappa$ and
$C_2$ only depending on $K$.
\end{theorem}

\begin{proof}
We shall need uniform three-spheres inequalities in a domain slightly larger than $G$.
For instance, for any $x\in G^{\frac{h}{2}}$, we have $B_{\frac{h}{2}}(x)\subset\Omega$.
Therefore we can apply the three-spheres inequality
\eqref{3spheres_nonhomogeneous} to spheres centered at $x$ choosing $R=h/2$. Hence $R_0=\min\{\frac{h}{2},C_0\rho_0\}$ and, by our choice in the assumptions, $R_0=\frac{h}{2}$. Moreover, recalling that $C_0\leq 1$, we also have
$\max\left\{\frac{R_0}{\rho_0},1\right\}=1$.

Next we can fix radii $r_1$, $r_2$, $r_3$ as follows
\begin{equation}\label{r3def}
r_3=\frac{h}{2},\quad r_2=\frac{r_3}{5K}=\frac{h}{10K},\quad r_1=\frac{1}{3}r_2=\frac{h}{30K}.
\end{equation}
With such a choice the inequality
\eqref{3spheres_nonhomogeneous} applies with $Q\geq 1$ and $\alpha$, $0<\alpha<1$,
only depending on $K$, $L$ and $\kappa$.

Let us consider the set $G^{r_1}$ as defined in \eqref{eq:est_env}. We have that $G^{r_1}$ is a
connected open set containing $G$ such that $\mathrm{dist}(G^{r_1},\partial\Omega)\geq h-r_1>r_3=\frac{h}{2}$.
For every $y\in G^{r_1}$, there exists a continuous path
$\gamma:[0.1]\to G^{r_1}$ such that $\gamma(0)=x_0$, $\gamma(1)=y$. Let us define
$0=t_0<t_1<\ldots<t_N=1$, according to the following rule. We set $t_{k+1}=\max\{t\ |\ |\gamma(t)-x_k|=2r_1\}$ if
$|x_k-y|>2r_1$, otherwise we stop the process and set $N=k+1$, $t_N=1$. Let $x_k=\gamma(t_k)$. The balls
$B_{r_1}(x_k)$ are pairwise disjoint for $k=0,\ldots,N-1$ and $|x_{k+1}-x_k|=2r_1$. Since $r_2=3r_1$ we have that
$B_{r_1}(x_{k+1})\subset B_{r_2}(x_k)$ and therefore, by the three-spheres inequality \eqref{3spheres_nonhomogeneous},
\begin{equation}
    \label{PS_int1}
    \|u\|_{L^2(B_{r_1}(x_{k+1}))}+\varepsilon\leq Q \left(\|u\|_{L^2(B_{r_1}(x_k))}+\varepsilon\right)^\alpha
    \left(E_0+\varepsilon\right)^{1-\alpha},
\end{equation}
for $k=0,\ldots,N-1$, where $Q\geq 1$ and $\alpha$, $0<\alpha<1$, only depend on $K$, $L$ and $\kappa$.

Denoting
\begin{equation*}
m_k=\frac{\|u\|_{L^2(B_{r_1}(x_k))}+\varepsilon}{E_0+\varepsilon},
\end{equation*}
we then have
\begin{equation}
    \label{PS_int2}
    m_{k+1}\leq Qm_k^\alpha, \quad\hbox{for } k=0,\ldots,N-1,
\end{equation}
\begin{equation}
    \label{PS_int3}
    m_N\leq \widetilde{Q}m_0^{\delta},
\end{equation}
where $\widetilde{Q}=Q^{1+\alpha+\ldots+\alpha^{N-1}}$ and $\delta=\alpha^N$.
Hence we have obtained
\begin{equation}
    \label{PS_int4}
    \|u\|_{L^2(B_{r_1}(y))}\leq \widetilde{Q}\left(\|u\|_{L^2(B_{r_1}(x_0))}+\varepsilon\right)^{\delta}(E_0+\varepsilon)^{1-\delta}.
\end{equation}
Now, $1+\alpha+\ldots+\alpha^{N-1}\leq \frac{1}{1-\alpha}$. Since $B_{r_1}(x_0)$,\ldots,$B_{r_1}(x_{N-1})$
are pairwise disjoint, we have that
$N\leq \frac{|\Omega|}{\omega_nr_1^n}\leq \frac{C_2|\Omega|}{h^n}$, with $C_2$ only depending on
$K$. Hence, recalling that $Q\geq 1$, we compute
\begin{equation}\label{Qtildeest}
\widetilde{Q}\leq Q^{\frac{1}{1-\alpha}},
\end{equation}
\begin{equation}\label{deltaest}
\delta\geq \alpha^{\frac{C_2|\Omega|}{h^n}}.
\end{equation}

Let us tessellate $\R^n$ with internally non-overlapping closed cubes of side $l=\frac{2r_1}{\sqrt n}$ and let $Q_j$, $j=1,\ldots,J$,
be those cubes which intersect $G$. Clearly, any such cube is contained in a ball of radius $r_1$ and center
$w_j\in G^{r_1}$ and $J\leq \frac{n^{\frac{n}{2}}|\Omega|}{2^nr_1^n}$.
Therefore, from \eqref{PS_int4}, we have
\begin{equation}
    \label{PS_int5}
    \int_G u^2\leq\sum_{j=1}^J\int_{Q_j}u^2\leq\sum_{j=1}^J\int_{B_{r_1}(w_j)}u^2\leq J\widetilde{Q}^2\rho_0^n\left(\|u\|_{L^2(B_{r_1}(x_0))}+\varepsilon\right)^{2\delta}(E_0+\varepsilon)^{2(1-\delta)}.
\end{equation}
Hence the thesis immediately follows.\end{proof}

\bigskip



\begin{remark}
It is important at this stage to emphasize that the above Theorem~\ref{theo:PSinterior}
on propagation of smallness in the interior enables us to generalize the three-spheres
inequality of Theorem~\ref{cor:3spheres_nonhomogeneous}, by removing the limitations on the radii that were present
there. It should also be mentioned however that by this approach it does not seem possible
to obtain optimal estimations of the constants $C$, $\alpha$ appearing in the inequality
\eqref{3spheres_nonhomogeneous_2}. This optimization problem for such a general version of the
three-spheres inequality still mantains some unanswered aspects.
\end{remark}

\begin{proof}[Proof of Theorem~\ref{cor:3spheres_nonhomogeneous_2}]
It follows by applying Theorem~\ref{theo:PSinterior}, with $B_{r_0}(x_0)=B_{r_1}$, $G=B_{r_2}$, $\Omega=B_{r_3}$.

The dependence of $C$, $\alpha$ on the quantities stated in Theorem~\ref{cor:3spheres_nonhomogeneous_2} is straightforward,
although somewhat lengthy.\end{proof}

\bigskip

\begin{theorem}[Global propagation of smallness]
    \label{theo:PSglobal}
    Let $\Omega$ be a bounded connected open set in $\R^n$ with boundary $\partial \Omega$ of Lipschitz class with constants $\rho_0$, $M_0$.
    Let $u\in H^1(\Omega)$ be a solution to the equation
\begin{equation}
    \label{div_equation_fF_ter}
    \divrg(A\nabla u)+cu=f+\divrg F, \quad\hbox{in } \Omega,
\end{equation}
where $f$ and $F$ satisfy \eqref{bound_fF}.
Let $B_{r_0}(x_0)\subset \Omega$ and let us assume that
\begin{equation}
    \label{sigma}
    \|u\|_{L^2(B_{r_0}(x_0))}\leq\eta,
\end{equation}
\begin{equation}
    \label{E}
    \|u\|_{H^1(\Omega)}\leq E,
\end{equation}
for given $\eta>0$, $E>0$. We have
\begin{equation}
    \label{PSglobal}
    \|u\|_{L^2(\Omega)}\leq (E+\varepsilon)\omega\left(\frac{\eta+\varepsilon}{E+\varepsilon}\right),
\end{equation}
where
\begin{equation}
    \label{omega}
    \omega(t)\leq \frac{C}{\left(\log\frac{1}{t}\right)^\mu},\quad \hbox{for  }t<1.
\end{equation}
where $C>0$ and $\mu$, $0<\mu<1$, only depend on $K$, $L$, $\kappa$, $M_0$, $\frac{r_0}{\rho_0}$ and
$\frac{|\Omega|}{\rho_0^n}$.
\end{theorem}

\begin{remark}
As it will be evident from the proof, rather than the bound \eqref{E} we shall actually
use a weaker one. In fact we might replace \eqref{E} by the assumption that there exists a $p>2$ such that
\begin{equation}\label{Ep}
\|u\|_{L^p(\Omega)}\leq E.
\end{equation}
In this case the constants $C$, $\mu$ would also depend on such an exponent $p>2$. We have chosen to formulate Theorem~\ref{theo:PSglobal} (and also Theorem~\ref{theo:Cauchy_global}) in terms of an a-priori $H^1$-bound because in many applied settings such a bound has a clearer physical interpretation.
\end{remark}

We premise the following proposition.
\begin{proposition}
   \label{prop:connected}
Let $\Omega$ be a bounded connected open set in $\R^n$ with boundary $\partial \Omega$
of Lipschitz class with constants $\rho_0$, $M_0$. There exists $h_0>0$, only depending on $\rho_0$, $M_0$, such that
$\Omega_h$ is connected for every $h<h_0$.
\end{proposition}
In order to prove the above proposition, let us introduce the following lemma.
We recall our assumptions.
$\Omega$ is a bounded connected open set in $\R^n$ with boundary $\partial \Omega$
of Lipschitz class with constants $\rho_0$, $M_0$. Let $P\in \partial\Omega$.
According to Definition~\ref{def:Lipschitz_boundary}, up to a rigid motion, we have
$P=0$ and
\begin{equation*}
  \Omega \cap \Gamma_{\frac{\rho_{0}}{M_0},\rho_0}(P)=\{x=(x',x_n) \in \Gamma_{\frac{\rho_{0}}{M_0},\rho_0}\quad | \quad
x_{n}>Z(x')
  \},
\end{equation*}
where $Z$ is a Lipschitz function on $B'_{\frac{\rho_{0}}{M_0}}
\subset {\R}^{n-1}$ satisfying
\begin{equation*}
Z(0)=0,
\end{equation*}
\begin{equation*}
\|Z\|_{{C}^{0,1}\left(B'_{\frac{\rho_{0}}{M_0}}\right)} \leq M_{0}\rho_{0}.
\end{equation*}
We denote
\begin{equation}
   \label{d(x)}
d(x)=\mathrm{dist}(x,\partial\Omega), \quad x\in\Omega,
\end{equation}
\begin{equation}
   \label{g(x)}
\widetilde{d}(x)=x_n-Z(x'), \quad x\in\Omega\cap \Gamma_{\frac{\rho_{0}}{M_0},\rho_0},
\end{equation}
and also
\begin{equation}
   \label{d_0}
   d_0=\frac{\rho_0}{2M_0}.
\end{equation}

\begin{lemma}
   \label{lem:distanze}
Under the above stated assumptions,
if we assume
\begin{equation}
   \label{d(x)less}
    x\in \Gamma_{\frac{\rho_0}{2M_0},\frac{\rho_0}{2}}\cap\Omega,\quad d(x)\leq d_0,
\end{equation}
then we have
\begin{equation}
   \label{g(x)lessd(x)}
    d(x)\leq \widetilde{d}(x)\leq \left(1+\sqrt{1+M_0^2}\right)d(x).
\end{equation}
\end{lemma}

\begin{proof}
Let $x\in \Gamma_{\frac{\rho_0}{2M_0},\frac{\rho_0}{2}}\cap\Omega$ be such that $d(x)\leq d_0$.
Let $z\in\partial\Omega$ be such that $|z-x|=d(x)$. Recalling that $M_0\geq 1$, we have that
\begin{equation*}
     |z'|\leq |z'-x'|+|x'|< d(x)+\frac{\rho_0}{2M_0}\leq \frac{\rho_0}{M_0},
\end{equation*}
\begin{equation*}
|z_n|\leq |z_n-x_n|+|x_n|<d(x)+\frac{\rho_0}{2}\leq \frac{\rho_0}{2M_0}+\frac{\rho_0}{2}\leq \rho_0.
\end{equation*}
Therefore $z\in \Gamma_{\frac{\rho_0}{M_0},\rho_0}\cap\partial\Omega$, so that $z=(z',Z(z'))$.

Let $y=(x',Z(x'))$. We have that
\begin{equation*}
    d(x)=|x-z|\leq |x-y|=\widetilde{d}(x),
\end{equation*}
\begin{multline*}
    \widetilde{d}(x)=|x-y|\leq |x-z|+|z-y|= |x-z|+\left(|x'-z'|^2+|Z(x')-Z(z')|^2\right)^{\frac{1}{2}}\leq
    \\
    \leq |x-z|+\sqrt{1+M_0^2}|x'-z'|\leq \left(1+\sqrt{1+M_0^2}\right)d(x).
\end{multline*}
\end{proof}

\bigskip

\begin{proof}[Proof of Proposition~\ref{prop:connected}]
Let $h_0=\frac{\rho_0}{4M_0\left(1+\sqrt{1+M_0^2}\right)}$.
We trivially have that $h_0<\frac{\rho_0}{8M_0}$.
Let $h< h_0$.
Let us first prove that $\overline{\Omega_h}$ is connected.
Given $x$, $y\in \overline{\Omega_h}$, let $\gamma$ be a path in $\Omega$ connecting $x$ and $y$.
If $\gamma$ is not contained in $\overline{\Omega_h}$, let us modify $\gamma$ to a new path
$\widetilde{\gamma}$ connecting $x$ and $y$ and contained in $\overline{\Omega_h}$.
Let $C=\{\gamma(t)\ |\ d(\gamma(t))\leq h\}$.
For every $z\in C$, there exists $\xi_z\in\partial \Omega$ such that $|z-\xi_z|\leq h$ so that
$z\in \Gamma_{\frac{\rho_0}{8M_0},\frac{\rho_0}{8}}(\xi_z)$, where $\Gamma_{\frac{\rho_0}{8M_0},\frac{\rho_0}{8}}(\xi_z)$
is defined according to the Definition~\ref{def:Lipschitz_boundary}. Therefore the
cylinders $\Gamma_{\frac{\rho_0}{8M_0},\frac{\rho_0}{8}}(\xi_z)$ provide an open covering of $C$ and, since $C$ is compact, there exist $z_1,\ldots,z_N$ such that
$C\subset \cup_{i=1}^N \Gamma_{\frac{\rho_0}{8M_0},\frac{\rho_0}{8}}(\xi_{z_i})$. Let us denote for simplicity $\xi_i=\xi_{z_i}$, $\Gamma^i=\Gamma_{\frac{\rho_0}{8M_0},\frac{\rho_0}{8}}(\xi_i)$.
Note that each such cylinder is possibly oriented with respect to a different coordinate system.
For every point $\gamma(t)\in C\cap \overline{\Gamma^1}$
let us replace
$\gamma(t)$ with $\widetilde{\gamma}(t)=(\gamma'(t),Z(\gamma'(t))+(1+\sqrt{1+M_0^2})h)$.
We can apply Lemma~\ref{lem:distanze} to $\gamma(t)$, obtaining that
$\gamma_n(t)-Z(\gamma'(t))=\widetilde{d}(\gamma(t))\leq (1+\sqrt{1+M_0^2})h$.
Recalling that $M_0\geq 1$, we have that
\begin{equation*}
   |\widetilde{\gamma}'(t)-(\xi_1)'|=|\gamma'(t)-(\xi_1)'|\leq \frac{\rho_0}{8M_0}<\frac{\rho_0}{2M_0},
\end{equation*}
\begin{equation*}
   |\widetilde{\gamma}_n(t)-\xi^1_n|\leq|\widetilde{\gamma}_n(t)-\gamma_n(t)|+|\gamma_n(t)-\xi^1_n|\leq
   \left(1+\sqrt{1+M_0^2}\right)h+\frac{\rho_0}{8}<\frac{\rho_0}{2}.
\end{equation*}
Therefore $\widetilde{\gamma}(t)\in \Gamma_{\frac{\rho_0}{2M_0},\frac{\rho_0}{2}}(\xi^1)\cap\Omega$ and $d(\widetilde{\gamma}(t))\leq|\widetilde{\gamma}(t)-(\gamma'(t),Z(\gamma'(t))|\leq (1+\sqrt{1+M_0^2})h< d_0$.
We can therefore apply Lemma~\ref{lem:distanze} to $\widetilde{\gamma}(t)$ and, by \eqref{g(x)lessd(x)},
\begin{equation*}
   d(\widetilde{\gamma}(t))\geq \frac{\widetilde{d}(\widetilde{\gamma}(t))}{1+\sqrt{1+M_0^2}}=h.
\end{equation*}
The connected components of $\gamma^{-1}(C\cap \overline{\Gamma^1})$ are closed intervals $I_\alpha$.
In order to glue together $\gamma $ and $\widetilde{\gamma}$, let us add, for each endpoint $t$ of any $I_\alpha$
and for any $t\in \gamma^{-1}(C\cap \partial\Gamma^1)$,
the closed segment joining $(\gamma'(t),\gamma_n(t))$ and $(\gamma'(t),Z(\gamma'(t))+(1+\sqrt{1+M_0^2})h)$.
Now we repeat the above arguments to the path so modified and to the cylinder $\Gamma^2$,
and so on for a finite number of steps.

We can now conclude proving that $\Omega_h$ is connected for every $h<h_0$.
Let $x$, $y\in\Omega_h$ and let $h'$ such that $h<h'<\min\{d(x),d(y), h_0\}$.
We have just shown that $\overline{\Omega}_{h'}$ is path connected, hence there exists a path in
$\overline{\Omega_{h'}}\subset\Omega_h$ joining $x$ and $y$.
\end{proof}

\bigskip

\begin{lemma}
   \label{lem:AR}
Let $\Omega$ be a bounded connected open set in $\R^n$ with boundary $\partial \Omega$
of Lipschitz class with constants $\rho_0$, $M_0$.
There exists a constant $C>0$, only depending on $M_0$, such that
\begin{equation}
   \label{AR}
    \mu(\Omega\setminus \Omega_h)\leq C|\Omega|\frac{h}{\rho_0}.
\end{equation}
\end{lemma}
\begin{proof}
We refer to \cite[(A.3)]{alrossSIAM} for a detailed proof. In fact this result was proved there under stronger regularity assumptions,
but the same arguments apply also under our Lipschitz regularity assumption.\end{proof}

\bigskip

\begin{proof}[Proof of Theorem~\ref{theo:PSglobal}]
Let $h_1=\min\left\{\frac{h_0}{2},\frac{r_0}{2},2C_0\rho_0\right\}$, where
$h_0=\frac{\rho_0}{4M_0(1+\sqrt{1+M_0^2})}$ has been introduced in
Proposition~\ref{prop:connected}. We have that $\Omega_{h_1}$ is connected, it contains
$B_{\frac{r_0}{2}}(x_0)$ and also we can apply
Theorem~\ref{theo:PSinterior}, with $G=\Omega_{h_1}$. That is
\begin{equation}
    \label{PS_Omega_h1}
    \|u\|_{L^2(\Omega_{h_1})}\leq C\left(\|u\|_{L^2(B_{\frac{r_0}{2}}(x_0))}+\varepsilon\right)^{\delta}
   ( \|u\|_{L^2(\Omega)}+\varepsilon)^{1-\delta}\leq C(E+\varepsilon)\left(\frac{\eta+\varepsilon}{E+\varepsilon}\right)^\delta,
\end{equation}
where $C>0$ and $\delta$, $0<\delta<1$, only depend on $K$, $L$, $\kappa$, $M_0$, $\frac{\rho_0}{r_0}$, and $\frac{|\Omega|}{\rho_0^n}$.

Let $r\in (0,h_1)$. Given any point $x\in\Omega_r\setminus \Omega_{h_1}$, that is such that
$r<d(x)\leq h_1$, let $\xi\in\partial\Omega$ be such that $d=d(x)=|\xi-x|$.

According to the Definition~\ref{def:Lipschitz_boundary},
up to a rigid motion,
$\Omega \cap \Gamma_{\frac{\rho_{0}}{M_0},\rho_0}(\xi)=
\{y=(y',y_n) \in \Gamma_{\frac{\rho_{0}}{M_0},\rho_0}\quad | \quad
y_{n}>Z(y')\}$, with $Z$ a Lipschitz function on $B'_{\frac{\rho_{0}}{M_0}}
\subset {\R}^{n-1}$ satisfying $Z(0)=0$, $\|Z\|_{{C}^{0,1}} \leq M_{0}\rho_{0}$.

By \eqref{g(x)lessd(x)}, we have that
$r<d\leq \widetilde{d}(x)=x_n-Z(x')\leq (1+\sqrt{1+M_0^2})d$.
Let $w=(x',Z(x'))\in\partial\Omega$. We have that $|x'|\leq d$, $|Z(x')|\leq M_0d$.

Let us translate the origin of the coordinate system {}from $\xi$ to $w$.

Let $t_0=\frac{\sqrt{1+M_0^2}}{1+\sqrt{1+M_0^2}}(\rho_0-M_0h_1)$,
$s_0=\frac{1}{4}\left(\frac{\rho_0-M_0h_1}{1+\sqrt{1+M_0^2}}-h_1\right)$ and $y_0=(0,t_0)$.
We have that
\begin{equation*}
    4s_0+h_1=\frac{t_0}{\sqrt{1+M_0^2}},
\end{equation*}
\begin{equation*}
    t_0+4s_0+h_1\leq \rho_0-M_0d,
\end{equation*}
\begin{equation*}
    4s_0\leq \frac{\rho_0}{M_0}-d,
\end{equation*}
so that
\begin{equation*}
    B_{4s_0}(y_0)\subset\Omega_{h_1}\cap \Gamma_{\frac{\rho_{0}}{M_0}-d,\rho_0-M_0d}(w)\subset\Omega_{h_1}\cap \Gamma_{\frac{\rho_{0}}{M_0},\rho_0}(\xi).
\end{equation*}

Let $C_w=\{y=(y',y_n)\ |\  y_n>\widetilde{M}|y'|\}$ be the open cone with vertex at $w$ and tangent to the ball
$B_{s_0}(y_0)$, that is $\widetilde{M}=\frac{1}{\tan\vartheta}$, with $\sin\vartheta=\frac{s_0}{t_0}$. Let us denote $s=\sin\vartheta$, $q=\frac{1-s}{1+s}$.
Let $s_1=qs_0$, $t_1=qt_0$ and $y_1=(0,t_1)$. Then the ball $B_{s_1}(y_1)$ is tangent to
the cone $C_w$ and to the ball $B_{s_0}(y_0)$. By induction, setting
$s_{k}=qs_{k-1}=q^{k}s_0$, $t_{k}=qt_{k-1}=q^{k}t_0$, $y_k=(0,t_k)$, for $k\geq 2$, we have that the ball $B_{s_k}(y_k)$ is tangent to
the cone $C_w$ and to the ball $B_{s_{k-1}}(y_{k-1})$.
Let
\begin{equation*}
    m_k=\frac{\|u\|_{L^2(B_{s_k}(y_k))}+\varepsilon}{E+\varepsilon}.
\end{equation*}
By applying \eqref{3spheres_nonhomogeneous_2} to the balls of center $y_k$ and radii $s_k$, $3s_k$, $4s_k$, and noticing that
$B_{s_{k+1}}(y_{k+1})\subset B_{3s_{k}}(y_{k})$, we have
\begin{equation}
    \label{PS_glo2}
    m_{k+1}\leq Cm_k^\alpha, \quad \hbox{for every }k=0,1,\ldots,
\end{equation}
\begin{equation}
    \label{PS_glo3}
    m_N\leq C^{1+\alpha+\ldots+\alpha^{N-1}}m_0^{\alpha^N}, \quad \hbox{for every } N=1,2,\ldots,
\end{equation}
where $C>0$ and $\alpha$, $0<\alpha<1$, only depend on $K$, $L$ and $\kappa$.
By
\eqref{PS_Omega_h1} we have
\begin{equation}
    \label{PS_m0}
    m_0\leq C\left(\frac{\eta+\varepsilon}{E+\varepsilon}\right)^{\delta},
\end{equation}
so that
\begin{equation}
    \label{PS_glo4}
    m_N\leq C \left(\left(\frac{\eta+\varepsilon}{E+\varepsilon}\right)^{\delta}\right)^{\alpha^N},
\end{equation}
where $\alpha\in (0,1)$ only depends on $K$, $L$ and $\kappa$, whereas $C>0$ and $\delta$, $0<\delta<1$, only depend on $K$, $L$, $\kappa$, $M_0$, $\frac{\rho_0}{r_0}$ and $\frac{|\Omega|}{\rho_0^n}$.

In the new coordinate system centered at $w$, $x=(0,\widetilde{d}(x))$ and, by \eqref{g(x)lessd(x)} and by the choice of $t_0$, $r<\widetilde{d}(x)\leq (1+\sqrt{1+M_0^2})d\leq (1+\sqrt{1+M_0^2})h_1<t_0$.
Since $x\not\in\Omega_{h_1}$, we have that $x\not\in B_{s_0}(y_0)\cup B_{s_1}(y_1)\subset B_{3s_0}(y_0)\subset \Omega_{h_1}$.
Hence there exists $N\in \N$, $N\geq 2$, such that $x\in \overline{B_{s_N}(y_N)}$, $x\not\in B_{s_{N-1}}(y_{N-1})$.
It follows that
$\widetilde{d}(x)<t_{N-1}$ so that $q^{N-1}>\frac{r}{t_0}$.

By \eqref{3spheres_nonhomogeneous_2} and by \eqref{PS_glo4}, we have
\begin{equation}
    \label{PS_glo5}
    \frac{\|u\|_{L^2(B_{3s_N}(y_N))}}{E+\varepsilon}\leq Cm_N^\alpha\leq
    C \left(\frac{\eta+\varepsilon}{E+\varepsilon}\right)^{\gamma\alpha^{N-1}},
\end{equation}
where $C>0$ and $\gamma =\alpha\delta$ only depend on $K$, $L$, $\kappa$, $M_0$, $\frac{\rho_0}{r_0}$ and $\frac{|\Omega|}{\rho_0^n}$.
Since $\alpha^{N-1}=q^{(N-1)\frac{\log\alpha}{\log q}}>\left(\frac{r}{t_0}\right)^{\frac{\log\alpha}{\log q}}$, and noticing that
$\frac{\eta+\varepsilon}{E+\varepsilon}<1$, we have
\begin{equation}
    \label{PS_glo6}
    \|u\|_{L^2(B_{3s_N}(y_N))}\leq
    C (E+\varepsilon)\left(\frac{\eta+\varepsilon}{E+\varepsilon}\right)^{\gamma\left(\frac{r}{t_0}\right)^D},
\end{equation}
with $D=\frac{\log \alpha}{\log q}$ and
$C>0$ and $\gamma\in (0,1)$ only depending on $K$, $L$, $\kappa$, $M_0$, $\frac{\rho_0}{r_0}$ and $\frac{|\Omega|}{\rho_0^n}$.
Moreover $s_N=q^Ns_0>\frac{rqs_0}{t_0}$. Since $B_{2s_N}(x)\subset B_{3s_N}(y_N)$,
we have that
\begin{equation}
    \label{PS_glo7}
    \|u\|_{L^2(B_{2s_N}(x))}\leq
    C (E+\varepsilon)\left(\frac{\eta+\varepsilon}{E+\varepsilon}\right)^{\gamma\left(\frac{r}{t_0}\right)^D},
\end{equation}
with $C>0$ and $\gamma\in (0,1)$ only depending on $K$, $L$, $\kappa$, $M_0$, $\frac{\rho_0}{r_0}$ and $\frac{|\Omega|}{\rho_0^n}$.
Let us tessellate $\R^n$ with closed cubes of side $l=\frac{2rqs_0}{\sqrt n t_0}$ and let $Q_j$, $j=1,\ldots,J$,
be those cubes which intersect $\Omega_r\setminus \Omega_{h_1}$. Clearly, any such cube is contained in a ball of radius $\frac{rqs_0}{t_0}$ and center
$w_j$ such that $|w_j-x_j|\leq \frac{rqs_0}{t_0}$, for some $x_j\in\Omega_r\setminus \Omega_{h_1}$.
By the above arguments, there exists $N\in\N$, $N\geq 2$, such that \eqref{PS_glo7} holds for $x=x_j$, with
$s_N>\frac{rqs_0}{t_0}$.
Therefore $B_{s_N}(w_j)\subset B_{2s_N}(x_j)$ and, by \eqref{PS_glo7} and by the trivial estimate $J\leq \frac{n^{\frac{n}{2}}t_0^n|\Omega|}{(2qrs_0)^n}$,
we have that
\begin{equation}
    \label{PS_glo8}
    \int_{\Omega_r\setminus \Omega_{h_1}} u^2\leq\sum_{j=1}^J\int_{Q_j}u^2\leq\sum_{j=1}^J\int_{B{s_N}(w_j)}u^2\leq C\rho_0^n (E+\varepsilon)^2\left(\frac{r}{t_0}\right)^{-n}
    \left(\frac{\eta+\varepsilon}{E+\varepsilon}\right)^{2\gamma\left(\frac{r}{t_0}\right)^D},
\end{equation}
where
$C>0$ and $\gamma\in (0,1)$ only depend on $K$, $L$, $\kappa$, $M_0$, $\frac{\rho_0}{r_0}$ and $\frac{|\Omega|}{\rho_0^n}$.
By \eqref{PS_Omega_h1} and by \eqref{PS_glo8}, we arrive at
\begin{equation}
    \label{PS_glo9}
    \|u\|_{L^2(\Omega_r)}\leq C (E+\varepsilon)\left(\frac{r}{t_0}\right)^{-\frac{n}{2}}\left(\frac{\eta+\varepsilon}{E+\varepsilon}\right)^{\gamma\left(\frac{r}{t_0}\right)^D},
\end{equation}
with $C>0$ and $\gamma\in (0,1)$ only depending on $K$, $L$, $\kappa$, $M_0$, $\frac{\rho_0}{r_0}$ and $\frac{|\Omega|}{\rho_0^n}$.
Let us choose $p=\frac{2n}{n-2}$, for $n>2$, whereas for $n=2$, let us choose as $p$ any number satisfying $p>2$.
{From} H\"older
inequality, Sobolev embedding theorem and by Lemma~\ref{lem:AR}, we deduce
\begin{multline}
  \label{eq:sobolev}
\|u\|_{L^2(\Omega\setminus\Omega_{r})}\leq
\left(\frac{|\Omega\setminus\Omega_r|}{\rho_0^n}\right)
^{\frac{1}{2}-\frac{1}{p}}\|u\|_{L^p(\Omega)}\leq \\
\leq C\left(\frac{|\Omega\setminus\Omega_r|}{\rho_0^n}\right)^{\frac{1}{2}-\frac{1}{p}}
\rho_0\|\nabla u\|_{L^2(\Omega)}\leq CE\left(\frac{r}{\rho_0}\right)^{\frac{1}{2}-\frac{1}{p}},
\end{multline}
with $C$ only depending on $M_0$ and $\frac{|\Omega|}{\rho_0^n}$. Recall that $t_0<\rho_0$, therefore we may replace $\rho_0$ with
$t_0$ in \eqref{eq:sobolev}.
By \eqref{PS_glo9} and \eqref{eq:sobolev}, we have that for every $r$, $0<r< h_1$,
\begin{equation}
    \label{PS_glo10}
    \|u\|_{L^2(\Omega)}\leq C (E+\varepsilon)\left(\frac{r}{t_0}\right)^{-\frac{n}{2}}\left(\left(\frac{\eta+\varepsilon}{E+\varepsilon}\right)^{\gamma\left(\frac{r}{t_0}\right)^D}
    +\left(\frac{r}{t_0}\right)^{\frac{1}{2}-\frac{1}{p}}\right),
\end{equation}
with $C>0$ and $\gamma\in (0,1)$ only depending on $K$, $L$, $\kappa$, $M_0$, $\frac{\rho_0}{r_0}$ and $\frac{|\Omega|}{\rho_0^n}$.
Setting
\begin{equation*}
    \tau=\left(\frac{r}{t_0}\right)^D,
\quad   \tau_0=\left(\frac{h_1}{t_0}\right)^D,
\end{equation*}
\begin{equation*}
    \vartheta=\frac{1}{D}\left(\frac{1}{2}-\frac{1}{p}\right),
\quad    \sigma=\frac{n}{2D},
\quad    \zeta=\left(\frac{\eta+\varepsilon}{E+\varepsilon}\right)^\gamma,
\end{equation*}
we obtain
\begin{equation}
    \label{PS_glo11}
    \|u\|_{L^2(\Omega)}\leq C(E+\varepsilon)(\tau^\vartheta+\tau^{-\sigma}\zeta^\tau), \quad \hbox{ for every }\tau, \ 0<\tau\leq \tau_0,
\end{equation}
with $C>0$ and $\gamma\in (0,1)$ only depending on $K$, $L$, $\kappa$, $M_0$, $\frac{\rho_0}{r_0}$ and $\frac{|\Omega|}{\rho_0^n}$.
Denoting
\begin{equation*}
    \Phi(\tau)=\tau^\vartheta+\tau^{-\sigma}\zeta^\tau,
\end{equation*}
let us estimate from above $\underset{0<\tau\leq \tau_0}{\inf}\Phi(\tau)$. To this aim, it is convenient to introduce
a new parameter $l$, related to $\tau$ by
\begin{equation*}
    \tau=\left(\frac{1}{\log\frac{1}{\zeta}}\right)^l.
\end{equation*}
We compute
\begin{equation*}
    \Phi(\tau)=\left(\frac{1}{\log\frac{1}{\zeta}}\right)^{l\vartheta}+\left(\frac{1}{\log\frac{1}{\zeta}}\right)^{-l\sigma}
    \exp\left\{-\left(\log\frac{1}{\zeta}\right)\frac{1}{\left(\log\frac{1}{\zeta}\right)^l}\right\}.
\end{equation*}
Let us temporarily assume $l<1$, and let us use the inequality $e^{-s}<1/s$ with
\begin{equation*}
    s=\frac{1}{\left(\log\frac{1}{\zeta}\right)^{l-1}}.
\end{equation*}
We obtain
\begin{equation*}
    \Phi(\tau)\leq\left(\frac{1}{\log\frac{1}{\zeta}}\right)^{l\vartheta}+\left(\frac{1}{\log\frac{1}{\zeta}}\right)^{1-l(1+\sigma)}.
\end{equation*}
Let us choose $l=\frac{1}{1+\vartheta+\sigma}$, so that $l\vartheta=1-l(1+\sigma)=\frac{\vartheta}{1+\vartheta+\sigma}$, and $0<l<1$.
Denoting
\begin{equation*}
    \mu=\min\{l\vartheta,1-l(1+\sigma)\}=\frac{\vartheta}{1+\vartheta+\sigma},
\end{equation*}
we have
\begin{equation*}
    \Phi(\tau)\leq 2\left(\frac{1}{\log\frac{1}{\zeta}}\right)^{\mu}.
\end{equation*}
Now, if $\left(\frac{1}{\log\frac{1}{\zeta}}\right)^{l}\leq \tau_0$, then
$\underset{0<\tau\leq \tau_0}{\min}\Phi(\tau)\leq 2\left(\frac{1}{\log\frac{1}{\zeta}}\right)^{\mu}$.
Otherwise, if $\left(\frac{1}{\log\frac{1}{\zeta}}\right)^{l}> \tau_0$, then
$\underset{0<\tau\leq \tau_0}{\min}\Phi(\tau)\leq \Phi(\tau_0)\leq C$, with $C>0$ only depending on
$K$, $L$, $\kappa$, $M_0$ and $\frac{r_0}{\rho_0}$.

Furthermore, $\left(\frac{1}{\log\frac{1}{\zeta}}\right)^{\mu}> \tau_0^{\frac{\mu}{l}}$, and hence
$\underset{0<\tau\leq \tau_0}{\min}\Phi(\tau)
\leq \frac{C}{\tau_0^{\frac{\mu}{l}}}\left(\frac{1}{\log\frac{1}{\zeta}}\right)^{\mu}$.
By \eqref{PS_glo11} and recalling the definition of $\zeta$,
we arrive at
\begin{equation}
    \label{PS_glo12}
    \|u\|_{L^2(\Omega)}\leq C (E +\varepsilon)\left(\frac{1}{-\gamma\log\frac{\eta+\varepsilon}{E+\varepsilon}}\right)^{\mu},
\end{equation}
where $C>0$ and $\mu$, $0<\mu<1$, only depend on $K$, $L$, $\kappa$, $M_0$, $\frac{r_0}{\rho_0}$ and
$\frac{|\Omega|}{\rho_0^n}$. Therefore, with a possibly new choice of $C$, we obtain
\begin{equation}
    \label{PS_glo13}
    \|u\|_{L^2(\Omega)}
    \leq C(E+\varepsilon) \left(\frac{1}{\log\frac{E+\varepsilon}{\eta+\varepsilon}}\right)^{\mu},
\end{equation}
where $C>0$ and $\mu$, $0<\mu<1$, only depend on $K$, $L$, $\kappa$, $M_0$, $\frac{r_0}{\rho_0}$ and
$\frac{|\Omega|}{\rho_0^n}$.
\end{proof}

\section{Proofs of the stability results for the Cauchy problem} \label{sec:
stability}

In order to reduce the problem of continuation from Cauchy data to a problem of continuation from an open set,
we shall introduce an augmented domain as follows.

We recall that $\Sigma$ is an open Lipschitz portion of $\partial\Omega$ with constants $\rho_0$, $M_0$ according to
Definition~\ref{def:Lipschitz_Sigma} and it has size at least $\rho_1>0$ according to Definition~\ref{def:size_Sigma}.

We denote $\eta:[0,+\infty)\to [0,1]$ as follows

\begin{equation}
  \label{eta}
  \eta(t)=
  \left\{ \begin{array}{ll}
  1, &
  0\leq t\leq \frac{1}{4},\\
  & \\
  4\left(\frac{1}{2}-t\right), &
  \frac{1}{4}\leq t\leq \frac{1}{2},\\
  & \\
  0, &
  t\geq \frac{1}{2}.\\
  \end{array}\right.
\end{equation}
Let $P\in\Sigma$ the point described in Definition~\ref{def:size_Sigma} and let $Z:B'_{\frac{\rho_0}{M_0}}\to \R$
be the Lipschitz function appearing in \eqref{Sigma_lip} for the local representation of $\Omega$ near $P$. According to \eqref{Sigma_lip} we have
\begin{equation*}
  \Omega \cap \Gamma_{\frac{\rho_{1}}{M_0},\rho_1}(P)=\{x=(x',x_n) \in \Gamma_{\frac{\rho_{1}}{M_0},\rho_1}\quad | \quad
x_{n}>Z(x')
  \}.
\end{equation*}
Let us denote
\begin{equation}
  \label{psi-}
  Z^-(x')=Z(x')-\frac{\rho_1}{2}\eta\left(\frac{M_0|x'|}{\rho_1}\right),\quad\hbox{for every }x'\in B'_{\frac{\rho_0}{M_0}}.
\end{equation}
Observe that
\begin{equation}
  \label{bound_psi-}
  |Z^-(x')|\leq |Z(x')|+\frac{\rho_1}{2},\quad\hbox{for every }x'\in B'_{\frac{\rho_0}{M_0}},
\end{equation}
\begin{equation}
  \label{bound_grad_psi-}
  |\nabla_{x'}Z^-(x')|\leq |\nabla_{x'}Z(x')|+2M_0,\quad\hbox{for every }x'\in B'_{\frac{\rho_0}{M_0}},
\end{equation}
and therefore, since $M_0\geq 1$,
\begin{equation}
  \label{norm_psi-}
  \|Z^-\|_{L^\infty\left(B'_{\frac{\rho_0}{M_0}}\right)}+
  \rho_0\|\nabla Z^-\|_{L^\infty\left(B'_{\frac{\rho_0}{M_0}}\right)}\leq
  \rho_0M_0+\left(\frac{\rho_1}{2}+2\rho_0\right)M_0\leq \frac{7}{2}\rho_0M_0.
\end{equation}
Next, we denote
\begin{equation}
  \label{open_A}
  \mathcal{A}=\{x=(x',x_n) \in \Gamma_{\frac{\rho_{0}}{M_0},\rho_0}\quad | \quad
  Z^-(x')<x_{n}<Z(x')
  \},
\end{equation}
\begin{equation}
  \label{Sigma_0}
  \Sigma_0=\{x=(x',x_n) \in \Gamma_{\frac{\rho_{0}}{M_0},\rho_0}\quad | \quad
  |x'|<\frac{\rho_{1}}{2M_0}, x_{n}=Z(x')
  \},
\end{equation}
\begin{equation}
  \label{Omega_tilde}
  \widetilde{\Omega}=\Omega\cup \Sigma_0\cup \mathcal{A}.
\end{equation}
It is a straightforward matter to verify that
\begin{description}
\item{i)}
if $\Omega$ has Lipschitz boundary with constants $\rho_0$, $M_0$, then $\widetilde{\Omega}$
has Lipschitz boundary with constants $\frac{\rho_0}{2}$, $\frac{7}{2}M_0$.
\item{ii)}
using the coordinates employed in the construction we have
\begin{equation}
  \label{Omega_tilde_contains}
  \widetilde{\Omega}\supset \Gamma_{\frac{\rho_{1}}{4M_0},\frac{\rho_1}{4}}.
\end{equation}
\item{iii)}
if we denote, for any $r>0$,
\begin{equation}
  \label{cono_sotto}
  C^-_r= \{x=(x',x_n)\in \R^n \ |\ -r<x_n<-M_0|x'|\},
\end{equation}
\begin{equation}
\label{cono_sopra}
  C^+_r= \{x=(x',x_n)\in \R^n \ |\ M_0|x'|<x_n< r\},
\end{equation}
we also have
\begin{equation}
  \label{A_contains}
  \mathcal{A}\supset C^-_{\frac{\rho_1}{4}}.
\end{equation}
\item{iv)}
if we denote
\begin{equation}
  \label{r1}
  r_0=\frac{\rho_1\left(\sqrt{1+M_0^2}-1\right)}{8M_0^2}=\frac{\rho_1}{8\left(\sqrt{1+M_0^2}+1\right)},
\end{equation}
\begin{equation}
  \label{P-}
  x_0=\left(0,r_0-\frac{\rho_1}{8}\right),
\end{equation}
we obtain that
\begin{equation}
  \label{ball_in cylinder}
  B_{r_0}(x_0)\subset C^-_{\frac{\rho_1}{8}}
\end{equation}
\end{description}

In the next lemma we continue to use the coordinate system centered in $P$ and the notation described in Definition~\ref{def:Lipschitz_boundary} and used above.

We denote
\begin{equation}
\Gamma^-_{\frac{\rho_1}{M_0},\rho_1}=\{(x',x_n)\in \Gamma_{\frac{\rho_1}{M_0},\rho_1}\ |\ x_n<Z(x')\}
\end{equation}
that is
\begin{equation}
\Gamma^-_{\frac{\rho_1}{M_0},\rho_1}=\Gamma_{\frac{\rho_1}{M_0},\rho_1}\setminus \overline{\Omega}.
\end{equation}

\begin{lemma}\label{lemma:extension}
Let $g\in H^{\frac{1}{2}}(\Sigma\cap \Gamma_{\frac{\rho_1}{M_0},\rho_1})$.
Then there exists $v\in H^1(\Gamma^-_{\frac{\rho_1}{M_0},\rho_1})$ such that
\begin{equation}\label{extensionv}
v|_{\Sigma\cap \Gamma_{\frac{\rho_1}{M_0},\rho_1}}=g\quad\text{in the sense of traces},
\end{equation}
and
\begin{equation}\label{extensionvest}
\|v\|_{H^1(\Gamma^-_{\frac{\rho_1}{M_0},\rho_1})}
\leq C\|g\|_{H^{\frac{1}{2}}(\Sigma\cap \Gamma_{\frac{\rho_1}{M_0},\rho_1})},
\end{equation}
where $C>0$ only depends on $M_0$ and $\frac{\rho_0}{\rho_1}$.
\end{lemma}

\begin{proof} This is a well-known fact. It suffices to prove it first in the reference situation when $Z\equiv 0$, $M_0=1$ and $\rho_1=1$, see for instance
\cite[Lemma~6.9.1]{KJF}.

Next, by a scaling, and by our convention on norms as described in Remark~\ref{rem:normal_norm}, we obtain \eqref{extensionvest} when $Z\equiv 0$, $M_0=1$
and $\rho_1>0$ is arbitrary. At this stage the constant $C$ in \eqref{extensionvest}
shall depend on $\frac{\rho_0}{\rho_1}$ only.

Finally, by a bilipschitz change of coordinates we may pass to the general (non-flat) case, at the price of admitting that $C$ also depends on $M_0$.\end{proof}

\bigskip

Let us now define
\begin{equation}\label{utilde}
\widetilde{u}=\left\{\begin{array}{ll}u, &\text{in }\Omega,\\
v, &\text{in }\Gamma^-_{\frac{\rho_1}{M_0},\rho_1},
\end{array}\right.
\end{equation}
where $u\in H^1(\Omega)$ is a weak solution to the Cauchy Problem~\ref{weakCauchy} and
$v$ is the function introduced in the previous lemma.

We denote
\begin{equation}
\Omega_1=\Omega\cup (\Sigma\cap \Gamma_{\frac{\rho_1}{M_0},\rho_1})
\cup \Gamma^-_{\frac{\rho_1}{M_0},\rho_1}.
\end{equation}

We obtain immediately that
\begin{equation}
\widetilde{u}\in H^1(\Omega_1)
\end{equation}
and also the following extension theorem.

\begin{theorem}[Extension]\label{theo:extension}
There exist
$\widetilde{f}\in L^2(\Omega_1)$, $\widetilde{F}\in L^2(\Omega_1;\R^n)$ such that
\begin{equation}\label{tildefF}
\|\widetilde{f}\|_{L^2(\Omega_1)}+\frac{1}{\rho_0}\|\widetilde{F}\|_{L^2(\Omega_1;\R^n)}
\leq C\frac{\varepsilon+\eta}{\rho_0^2}
\end{equation}
and $\widetilde{u}$ satisfies in the weak sense
\begin{equation}\label{tildeuequ}
\mathrm{div}(A\nabla\widetilde{u})+c\widetilde{u}=\widetilde{f}+\mathrm{div}\widetilde{F},\quad\text{in }\Omega_1.
\end{equation}
Here $C>0$ only depends on $M_0$, $K$, $\kappa$ and $\frac{\rho_0}{\rho_1}$.
\end{theorem}

\begin{proof} Let $\varphi$ be an arbitrary test function in $H^1_0(\Omega_1)$,
with support compactly contained in $\Omega_1$.
Evidently $\varphi|_{\Omega}\in H^1_{co}(\Omega\cup\Sigma)$.

Denoting for simplicity $\Gamma^-=\Gamma^-_{\frac{\rho_1}{M_0},\rho_1}$,
we compute
\begin{multline}\label{equation_c}
-\int_{\Omega_1}(A\nabla \widetilde{u}\cdot\nabla\varphi-c\widetilde{u}\varphi)=
-\int_{\Omega}(A\nabla u\cdot\nabla\varphi-cu\varphi)
-\int_{\Gamma^-}(A\nabla v\cdot\nabla\varphi-cv\varphi)=\\
=-\int_{\Sigma}\psi\varphi+\int_{\Omega}(f\varphi-F\cdot\nabla\varphi)
-\int_{\Gamma^-}(A\nabla v\cdot\nabla\varphi-cv\varphi).
\end{multline}
Let us also denote $\Sigma_1=\Sigma\cap\Gamma_{\frac{\rho_1}{M_0},\rho_1}$ and let us set
\begin{equation}\label{Psidef}
\Psi(\varphi)=\int_{\Sigma}\psi\varphi=\rho_0^{n-1}\frac{1}{\rho_0^{n-1}}\int_{\Sigma}\psi\varphi.
\end{equation}
We have
\begin{equation}\label{Psiest}
|\Psi(\varphi)|\leq \rho_0^{n-1}\|\psi\|_{H^{-\frac{1}{2}}(\Sigma)}
\|\varphi|_{\Sigma_1}\|_{H^{\frac{1}{2}}(\Sigma_1)}\leq
C\rho_0^{n-2}\eta\|\varphi\|_{H^1_0(\Omega_1)}.
\end{equation}
Here $C>0$ is the constant
for the trace imbedding $H^1_0(\Omega_1) \hookrightarrow  H^{\frac{1}{2}}(\Sigma_1)$
which only depends on the Lipschitz character of $\Sigma_1$. Hence $C$ only depends on $M_0$ and $\frac{\rho_0}{\rho_1}$.

Therefore $\Psi\in H^{-1}(\Omega_1)$ and its norm is bounded as follows
\begin{equation}\label{Psinormest}
\|\Psi\|_{H^{-1}(\Omega_1)}\leq C\rho_0^{n-2}\eta,
\end{equation}
where $C>0$ is the same constant as above.

By the well-known Riesz representation theorem in Hilbert spaces, we can find
$f_1\in L^2(\Omega_1)$, $F_1\in L^2(\Omega_1;\R^n)$ such that
\begin{equation}\label{f_1est}
\rho_0\|f_1\|_{L^2(\Omega_1)}+\|F_1\|_{L^2(\Omega_1;\R^n)}\leq
\frac{\sqrt{2}}{\rho_0^{n-1}}\|\Psi\|_{H^{-1}(\Omega_1)}\leq
C\frac{\eta}{\rho_0},
\end{equation}
and
\begin{equation}\label{Psicaratt}
\Psi(\varphi)=\int_{\Omega_1}f_1\varphi-F_1\cdot\nabla \varphi,\quad\text{for every }\varphi\in H^1_0(\Omega_1).
\end{equation}

Note that the powers of $\rho_0$ appearing above are calculated according to
the fact that the appropriate scalar product for $H^1(\Omega_1)$,
as derived by our conventions on norms, is given by
$$\langle \varphi,\varphi'\rangle_{H^1(\Omega_1)}=\frac{1}{\rho_0^n}\int_{\Omega_1}\varphi\varphi'+\rho_0^2\nabla\varphi\cdot\nabla\varphi',$$
and, analogously, the $L^2(\Sigma)$-scalar product is to be meant as follows
$$\langle \varphi,\varphi'\rangle_{L^2(\Sigma)}=\frac{1}{\rho_0^{n-1}}\int_{\Sigma}\varphi\varphi'.$$

We define
\begin{equation}\label{tildefdef}
\widetilde{f}=\left\{\begin{array}{ll}
f-f_1,&\text{in }\Omega,\\
cv-f_1,&\text{in }\Gamma^-,
\end{array}\right.
\end{equation}
\begin{equation}\label{tildeFdef}
\widetilde{F}=\left\{\begin{array}{ll}
F-F_1,&\text{in }\Omega,\\
A\nabla v-F_1,&\text{in }\Gamma^-.
\end{array}\right.
\end{equation}
We obtain
\begin{equation}\label{tildefest}
\|\widetilde{f}\|_{L^2(\Omega_1)}\leq \|f\|_{L^2(\Omega)}+
\|f_1\|_{L^2(\Omega_1)}+\frac{\kappa}{\rho_0^2}\|v\|_{L^2(\Gamma^-)},
\end{equation}
\begin{equation}\label{tildeFest}
\|\widetilde{F}\|_{L^2(\Omega_1;\R^n)}\leq \|F\|_{L^2(\Omega;\R^n)}+
\|F_1\|_{L^2(\Omega_1;\R^n)}+K\|\nabla v\|_{L^2(\Gamma^-;\R^n)}.
\end{equation}
Hence
\begin{multline}\label{tildefFestfinal}
\|\widetilde{f}\|_{L^2(\Omega_1)}+\frac{1}{\rho_0}
\|\widetilde{F}\|_{L^2(\Omega_1;\R^n)}\leq \\ \leq
\left(
\|f\|_{L^2(\Omega)}+\frac{1}{\rho_0}\|F\|_{L^2(\Omega;\R^n)}\right)+
\left(\|f_1\|_{L^2(\Omega_1)}+\frac{1}{\rho_0}\|F_1\|_{L^2(\Omega_1;\R^n)}\right)+
\frac{C}{\rho_0^2}\|v\|_{H^1(\Gamma^-)},
\end{multline}
where $C>0$ only depends on $\kappa$ and $K$.
Thus, recalling Lemma~\ref{lemma:extension} and the assumptions on
$f$, $F$, $g$ and $\psi$, we deduce \eqref{tildefF} and from \eqref{equation_c}
we obtain \eqref{tildeuequ}.\end{proof}

\bigskip

\begin{remark}\label{rem:extension}
It is evident how, by the above theorem, we can translate a Cauchy problem
into a problem of propagation of smallness. The philosophy is as follows. Whenever
we are given $\rho_1>0$ and $P\in\Sigma$ such
that $\Gamma_{\frac{\rho_1}{M_0},\rho_1}(P)\cap \Sigma$ is a Lipschitz graph,
and once we have shown that $\widetilde{u}|_{\Gamma^-}=v$ is
"small" in $\Gamma^-_{\frac{\rho_1}{M_0},\rho_1}$, then we can propagate such
a "smallness" in a H\"older fashion to any compact subset of
$\Gamma_{\frac{\rho_1}{M_0},\rho_1}$. That is,
we obtain that $u=\widetilde{u}|_{\Omega}$ is also "small" in a H\"older fashion in a region
near $P$ inside $\Omega$. This rough argument is made precise in
Theorem~\ref{theo:Cauchyinterior}, which we are going to prove here below.

For technical reasons, we found it convenient to make use of the augmented domain
$\widetilde{\Omega}\subset\Omega_1$ rather than $\Omega_1$ itself. The advantage is that when $\Omega$ is assumed to be globally Lipschitz, the global Lipschitz regularity
of $\widetilde{\Omega}$ is proven in a more transparent way. This fact will be somewhat
helpful later on for the proof of Theorem~\ref{theo:Cauchy_global}.
\end{remark}

\begin{proof}[Proof of Theorem~\ref{theo:Cauchyinterior}]
Let us set $x_0$ and $r_0$ as in \eqref{P-} and in \eqref{r1}, respectively.
Let
$\overline{h}=\min\{2C_0\rho_0,\frac{r_0}{2}\}$, $C_0$ as in the thesis of
Theorem~\ref{cor:3spheres_zero_order}. Notice that, by \eqref{r1}, $\overline{h}<\frac{\rho_1}{8M_0}$.
Denote
\begin{equation}
\label{Gtilde}
\widetilde{G}=G\cup\Gamma_{\frac{\rho_{1}}{8M_0},\frac{\rho_{1}}{8}}.
\end{equation}
Recalling \eqref{Omega_tilde_contains}, it is easily verified that $\widetilde{G}\subset\widetilde{\Omega}$, it is connected and
also
\begin{equation}
  \label{dist_Gtilde_bordo}
  \mathrm{dist}(\widetilde{G},\partial\widetilde{\Omega})\geq\min\left\{h,\frac{\rho_1}{8M_0}\right\}=h.
\end{equation}
Moreover, by \eqref{ball_in cylinder},
\begin{equation}
  \label{dist_ball_inside}
  B_{\frac{r_0}{2}}(x_0)\subset \mathcal{A}\cap\Gamma_{\frac{\rho_{1}}{8M_0},\frac{\rho_{1}}{8}}\subset \widetilde{G}.
\end{equation}
Hence we can apply Theorem~\ref{theo:PSinterior} with $u$, $f$, $F$, $\Omega$, $G$ replaced with $\widetilde{u}$, $\widetilde{f}$,
$\widetilde{F}$, $\widetilde{\Omega}$, $\widetilde{G}$ respectively. By Lemma~\ref{lemma:extension} and by \eqref{utilde} we recall that
\begin{equation}
  \label{bound_norm_ball}
  \|\widetilde{u}\|_{L^2(B_{r_0}(x_0))}\leq C\eta,
\end{equation}
\begin{equation}
  \label{bound_norm_domain}
  \|\widetilde{u}\|_{L^2(\widetilde{\Omega})}\leq C(E_0+\eta),
\end{equation}
where $C>0$ only depends on $M_0$ and $\frac{\rho_0}{\rho_1}$.
Moreover, noticing that $\Omega\supset C^+_{\rho_1}$, it is immediate to estimate $|\widetilde{\Omega}|\leq C|\Omega|$, with $C>0$ an absolute constant. Thus the thesis follows.\end{proof}

\bigskip

\begin{remark}
  \label{rem:weaken_H1}
  Observe that in Theorem~\ref{theo:Cauchyinterior} the assumption $u\in H^1(\Omega)$ is not really necessary. This assumption is made
  just for the sake of simplicity. A more appropriate assumption would be
\begin{equation}
\label{weaken_H1}
u\in\bigcap_{\overline{K}\subset\Omega\cup\Sigma}H^1(K).
\end{equation}
\end{remark}
\begin{remark}
  \label{rem:Holder_stability_extension}
  As is evident {}from the proof, the stability of H\"older type could be obtained also on any connected
  subset $G$ of $\Omega$ having a positive distance $h$ {}from $\Sigma'=\partial\Omega\setminus \sigma$ (rather than {}from
  $\partial\Omega$). However, unless some additional assumption on the shape and regularity of $\partial\Sigma$ is made, it might not be possible to specify in a precise manner how the constant $C$ and the exponent $\delta$ appearing in \eqref{Cauchyinterior} behave with respect to $h$ as $h\to 0$. This is in fact an important issue in view of obtaining a global stability bound when no global regularity (Lipschitz) information on $\partial\Omega$ is available. See Section~\ref{sec:
stability_loglog} below for further discussion on this issue.
  \end{remark}

\begin{proof}[Proof of Theorem~\ref{theo:Cauchy_global}]
Similarly to what we did in the previous proof, we apply Theorem~\ref{theo:PSglobal} with
$u$, $f$, $F$, $\Omega$ replaced with $\widetilde{u}$, $\widetilde{f}$,
$\widetilde{F}$, $\widetilde{\Omega}$, respectively. As before, we consider the ball
$B_{r_0}(x_0)$ as defined in \eqref{r1}, \eqref{P-} and we use the fact that, by Lemma~\ref{lemma:extension}
\begin{equation}
  \label{bound_norm_ball_bis}
  \|\widetilde{u}\|_{L^2(B_{r_0}(x_0))}\leq C\eta,
\end{equation}
\begin{equation}
  \label{bound_norm_domain_bis}
  \|\widetilde{u}\|_{H^1(\widetilde{\Omega})}\leq C(E+\eta),
\end{equation}
with $C>0$ only depending on $M_0$ and $\frac{\rho_0}{\rho_1}$.
\end{proof}

\section{A generalization of the stability results for the Cauchy problem} \label{sec:
stability_loglog}
\begin{theorem}[Global stability for the Cauchy problem - generalization]
    \label{theo:Cauchyglobal_loglog}
    Let $u\in H^1(\Omega)$ be a solution to the Cauchy Problem~\ref{weakCauchy}, where
    $\Sigma$ satisfies the conditions in Definition~\ref{def:Lipschitz_Sigma} and Definition~\ref{def:size_Sigma}, $f\in L^2(\R^n)$ and $F\in L^2(\R^n; \R^n)$ satisfy \eqref{bound_fF} and $g\in H^{\frac{1}{2}}\left( \Sigma \right)$, $\psi \in H^{-\frac{1}{2}}\left( \Sigma\right)$ satisfy \eqref{bound_g_psi}. Let us assume that there exists a family $\{G_h\}$, $0<h\leq \overline{h}$, $\overline{h}$ as in Theorem~\ref{theo:Cauchyinterior},
    of connected open sets $G_h\subset\Omega$ satisfying the conditions
\eqref{dist_G_bordo}, \eqref{dist_P_G} and also
\begin{equation}
  \label{measure_small}
  |\Omega\setminus G_h|\leq Q\rho_0^n\left(\frac{h}{\rho_0}\right)^\vartheta,
\end{equation}
for given $Q$, $\vartheta>0$.
If, given $E>0$, $p>2$, we a-priori assume that
\begin{equation}
  \label{bound_E_stabilityloglog}
  \|u\|_{L^p(\Omega)}\leq E,
\end{equation}
   then we have
\begin{equation}
    \label{Cauchygloballoglog}
    \|u\|_{L^2(\Omega)}\leq
    (E+\varepsilon+\eta)\omega\left(\frac{\varepsilon+\eta}{e(E+\varepsilon+\eta)}\right),
\end{equation}
where
\begin{equation}
    \label{omega_bisloglog}
    \omega(t)\leq \frac{C}{\left(\log|\log t|\right)^{S}},\quad \hbox{for  }0<t< \frac{1}{e},
\end{equation}
where $C>0$ and $S$, $0<S<1$, only depend on $K$, $L$, $\kappa$, $M_0$, $\frac{\rho_0}{\rho_1}$,
$\frac{|\Omega|}{\rho_0^n}$, $Q$, $\vartheta$ and $p$.
\end{theorem}

\begin{remark}
  \label{rem:loglog_PS}
In a completely analogous fashion, a global estimate of propagation of smallness of \emph{log--log}-type in a general connected open set $\Omega$ could be stated and proved.
\end{remark}

\begin{proof}[Proof of Theorem~\ref{theo:Cauchyglobal_loglog}]
By H\"older inequality, and by our convention on norms (Remark~\ref{rem:normal_norm}), we have
\begin{equation}
  \label{eq:L2_Lp}
\|u\|_{L^2(\Omega)}\leq
\left(\frac{|\Omega|}{\rho_0^n}\right)
^{\frac{1}{2}-\frac{1}{p}}\|u\|_{L^p(\Omega)}\leq CE,
\end{equation}
where $C>0$ only depends on $\frac{|\Omega|}{\rho_0^n}$ and $p$.
Hence, by Theorem~\ref{theo:Cauchyinterior},
\begin{equation}
\|u\|_{L^2(G_h)}\leq
C_h(E+\varepsilon+\eta)\left(\frac{\varepsilon+\eta}{E+\varepsilon+\eta}\right)
^{\delta_h},
\end{equation}
where
\begin{equation}
   \label{Cdef_ter}
C_h= C_1\left(\frac{|\Omega|}{h^n}\right)^{\frac{1}{2}}\quad\text{and}\quad
\delta_h\geq \alpha^{\frac{C_2|\Omega|}{h^n}},
\end{equation}
with $\alpha\in (0,1)$ only depending on $K$, $L$ and $\kappa$, $C_1$ only depending on $K$, $L$, $\kappa$, $M_0$ and $\frac{\rho_0}{\rho_1}$, and
$C_2$ only depending on $K$.

We notice that
\begin{multline}
  \label{eq:L2(Gh)}
\|u\|_{L^2(G_h)}\leq
C_h(E+\varepsilon+\eta)\left(\frac{\varepsilon+\eta}{E+\varepsilon+\eta}\right)
^{\delta_h}=\\ =C_h e^{\delta_h}(E+\varepsilon+\eta)\left(\frac{\varepsilon+\eta}
{e(E+\varepsilon+\eta)}\right)
^{\delta_h}\leq
C_h e(E+\varepsilon+\eta)\left(\frac{\varepsilon+\eta}
{e(E+\varepsilon+\eta)}\right)
^{\delta_h}.
\end{multline}
This last step has no substantial importance, it is made only because, at some later stage, it will be convenient to have that the ratio $\frac{\varepsilon+\eta}
{e(E+\varepsilon+\eta)}$ is strictly less than $e^{-1}$.

Again using H\"older
inequality, we also have
\begin{equation}
  \label{eq:L2(Omega-Gh)}
\|u\|_{L^2(\Omega\setminus G_h)}\leq
\left(\frac{|\Omega\setminus G_h|}{\rho_0^n}\right)
^{\frac{1}{2}-\frac{1}{p}}E\leq E
\left(
Q\left(\frac{|\Omega|}{\rho_0^n}\right)^{\frac{\vartheta}{n}}
\left(\frac{h^n}{|\Omega|}\right)^{\frac{\vartheta}{n}}
\right)^{\frac{1}{2}-\frac{1}{p}}.
\end{equation}
Setting
\begin{equation}
\label{sdef}
s=\frac{h^n}{|\Omega|}, \quad 0<s\leq s_0=\frac{\overline{h}^n}{|\Omega|},
\end{equation}
we obtain
\begin{equation}
  \label{eq:L2(Omega)less}
\|u\|_{L^2(\Omega)}\leq
C(E+\varepsilon+\eta)\left(
s^{-\frac{1}{2}}\tau^{\alpha^{-C_2s}}+s^D
\right), \quad \hbox{for every } s\in(0,s_0],
\end{equation}
where
\begin{equation}
  \label{eq:tau_loglog}
\tau=\frac{\varepsilon+\eta}{e(E+\varepsilon+\eta)}\in (0,e^{-1}),
\end{equation}
\begin{equation}
D=\frac{\theta}{n}\left(\frac{1}{2}-\frac{1}{p}\right),
\end{equation}
and $C$ only depends on
$K$, $L$, $\kappa$, $M_0$, $\frac{\rho_0}{\rho_1}$,
$\frac{|\Omega|}{\rho_0^n}$, $Q$, $\vartheta$ and $p$.
Note that also $s_0$ only depends on $K$, $L$, $\kappa$, $M_0$, $\frac{\rho_0}{\rho_1}$, $\frac{|\Omega|}{\rho_0^n}$.
A simple calculation (see for instance \cite[Proof of Proposition~3.1]{ABRVpisa}) gives
\begin{equation}
  \label{eq:finalestimate}
\underset{0<s\leq s_0}{\inf}
\left(
s^{-\frac{1}{2}}\tau^{\alpha^{-C_2s}}+s^D
\right)\leq C(\log|\log\tau|)^{-S}, \quad \hbox{for every } \tau\in(0,e^{-1}),
\end{equation}
where $C$, $S>0$ only depend on
$K$, $L$, $\kappa$, $M_0$, $\frac{\rho_0}{\rho_1}$,
$\frac{|\Omega|}{\rho_0^n}$, $Q$, $\vartheta$ and $p$.
\end{proof}

\def\cprime{$'$} \def\cydot{\leavevmode\raise.4ex\hbox{.}} \def\cprime{$'$}
  \def\cprime{$'$} \def\cprime{$'$} \def\cprime{$'$} \def\cprime{$'$}

\end{document}